\documentclass[letterpaper, 10 pt, conference]{ieeeconf} 
\IEEEoverridecommandlockouts                             

\usepackage{graphics} 
\usepackage{epsfig} 
\usepackage{scrextend}
\usepackage{times} 
        \makeatletter
        \let\proof\@undefined
        \let\endproof\@undefined
        \makeatother
\usepackage{amsthm}
\usepackage{amsmath, amsfonts, bm, amssymb, tabularx}
\usepackage{latexsym, mathtools}
\usepackage{color}
\usepackage{subcaption}

\newtheorem{theorem}{Theorem}[section]
\newtheorem{lemma}{Lemma}[section]
\newtheorem{definition}{Definition}[section]
\newtheorem{remark}{Remark}[section]

\usepackage{flushend}

\usepackage{algorithm}
\usepackage{algpseudocode}
\usepackage{Shorthands}
\usepackage{url}
\usepackage{mathtools}

\newcommand{\leo}[1]{{\color{black}#1}}
 \usepackage{url}
 \usepackage{hyperref}\hypersetup{colorlinks=true, unicode=true, linkcolor=[rgb]{0.10,0.05,0.67}, citecolor=[rgb]{0.10,0.05,0.67}, filecolor=[rgb]{0.10,0.05,0.67}, urlcolor=[rgb]{0.10,0.05,0.67}}
 \usepackage{cite}

\title{\LARGE \bf
Multitask LQG Control: Performance and Generalization Bounds
}

\author{Leonardo F. Toso$^{\star1}$, Kasra Fallah$^{\star1}$, Charis Stamouli$^{\star2}$, George J. Pappas$^2$, and James Anderson$^1$
\thanks{$^\star$Equally contributed to this work.}
\thanks{$^1$Leonardo F. Toso, Kasra Fallah, and James Anderson are with the Department of Electrical Engineering at Columbia University, USA. $^2$Charis Stamouli and George J. Pappas are with the Department of Electrical and Systems Engineering, University of Pennsylvania, USA. 
}
\thanks{Correspondence to \texttt{leonardo.toso@columbia.edu}.}\thanks{ }
\thanks{
        {\tt\small }}%
}

\allowdisplaybreaks

\begin{document}

\maketitle
\thispagestyle{empty}
\pagestyle{empty}

\begin{abstract}
We study multitask learning for stochastic and partially observed control systems, focusing on the linear quadratic Gaussian (LQG) problem. Our goal is to learn a common stabilizing controller that generalizes across a distribution of systems and objectives. To this end, we leverage a history-dependent lifting that recasts the multitask LQG problem into an equivalent high-dimensional multitask LQR problem, allowing for the analysis of policy gradient methods. We show that learning a common lifted controller induces a heterogeneity bias which we characterize via a ``bisimulation function''. We establish performance and generalization guarantees that explicitly depend on such bisimulation-based heterogeneity measures. For model-free, we demonstrate that multitask learning reduces policy gradient estimation variance proportionally to the number of tasks in the training set.
\end{abstract}

\section{Introduction}

Modern control systems increasingly operate in settings where multiple agents interact and collaborate to learn control policies that generalize across system-specific variations, such as differences in dynamics and objectives. Examples include robotic fleets \cite{wang2023fleet}, autonomous driving \cite{kiran2021deep}, and distributed sensor networks  \cite{sinopoli2003distributed}. Our goal is to learn a common controller that generalizes across a distribution of systems and objectives under stochastic data and with incomplete state information, focusing on the linear quadratic Gaussian (LQG) control problem \cite{zhou_robust_1996}.

A recent line of work studies multitask learning for linear quadratic regulator (LQR) tasks \cite{wang2023model,stamouli2025policy,fujinami2025policy,toso2024meta,ye2024convergence}. These works focus on fully observed, deterministic systems and establish optimality bounds for policy gradient methods \cite{fazel2018global} when learning a common stabilizing controller. While the  LQR problem provides a fundamental benchmark for understanding learning in control systems, it relies on assumptions that are restrictive in practice. In practice, systems are affected by process and measurement noise and typically operate under partial observability \cite{zhou_robust_1996}. Such intricacies are more accurately captured by the LQG problem, yet extending multitask learning guarantees to it remains largely open.

Even in the single-system setting, policy optimization for LQG is fundamentally harder than for LQR. In the LQR setting, global convergence follows from favorable geometric properties of the cost, such as gradient dominance and smoothness \cite{fazel2018global,mohammadi2021convergence,gravell2020learning,hu2023toward}. These properties do not hold for dynamic output-feedback controllers as in LQG due to partial observability \cite{mohammadi2021lack, tang2021analysis}. Recent results \cite{zhao_globally_2023,guo2023imitation, fallah2025gradient} address this limitation via a history-dependent lifting that recasts partial observability into a higher-dimensional fully observed history state. In particular, \cite{fallah2025gradient} leverages this construction to recast the LQG problem as an equivalent high-dimensional LQR problem, restoring a favorable optimization landscape and allowing for global convergence guarantees.

We extend this input-output history representation to the multitask LQG problem. We consider a collection of partially observed and noisy systems with quadratic costs and use the history representation to obtain an equivalent high-dimensional multitask LQR problem. This allows for the analysis of policy optimization to learn a common stabilizing controller across systems. A key challenge is that the multitask objective induces  system and cost heterogeneity biases \cite{toso2024async, stamouli2025policy}, that are not present in the single-system setting.

In particular, differences in system dynamics and cost parameters induce a mismatch in task-specific policy gradients, preventing a single controller from being simultaneously optimal for all systems. Prior work in multitask LQR \cite{stamouli2025policy} has characterized this effect through a \emph{closed-loop bisimulation-based} heterogeneity measure. Here, we use bisimulation functions to characterize system and cost heterogeneity in the multitask LQG setting. However, partial observability introduces additional difficulties, as heterogeneity now depends on both control and estimation dynamics in the lifted space.

Moreover, we study generalization and sample efficiency. We provide generalization guarantees for unseen tasks that explicitly depend on system and cost heterogeneity. In the model-free setting, gradient estimation relies on one-point zeroth-order estimators \cite{malik2019derivative}, which have higher variance than the two-point estimators available in noise-free LQR. We show that multitask learning mitigates this by reducing the variance proportionally to the number of training tasks.

\vspace{0.1cm}

\noindent \textbf{Contributions.} Our main contributions are as follows:

\noindent $\bullet$ We establish policy gradient bounds for the multitask LQG by leveraging a history-dependent lifting \cite{fallah2025gradient} that recasts the problem as high-dimensional multitask LQR. This provides the first such guarantees for stochastic and partially observed multitask LQG control.

\noindent $\bullet$ We introduce a bisimulation-based characterization of task heterogeneity tailored to the LQG problem. Our heterogeneity measure captures task discrepancies in both control and estimation dynamics through the lifted gradient dynamical system, building on prior work for multitask LQR \cite{stamouli2025policy}.

\noindent $\bullet$ We derive generalization bounds for multitask LQG that explicitly quantify how task heterogeneity impacts performance on unseen systems and objectives. Our results reveal the trade-off between collaboration and heterogeneity bias.

\noindent $\bullet$ In the model-free setting, we analyze one-point zeroth-order policy gradient estimators and show that multitask learning reduces the estimation variance proportionally to the number of tasks in the training set. 

\vspace{0.05cm}
All proofs are provided in the appendix.\vspace{0.1cm}

\noindent \textbf{Related Work.} Most relevant to this work are the results on multitask learning for control \cite{wang2023model,stamouli2025policy,fujinami2025policy,toso2024meta,ye2024convergence, zhan2025coreset, fallah2025adversarially}, which focus on the LQR problem. They establish optimality bounds under fully observed, deterministic dynamics. In contrast, we consider stochastic and partially observed systems and provide performance bounds for multitask LQG that depend on task heterogeneity that is characterized via approximate bisimulation \cite{girard2011approximate,  stamouli2025policy}.

In addition, generalization in multitask and meta-learning has been extensively studied in \cite{ji2022theoretical,liu2022theoretical,schnitzer2026probabilistic}, where transfer across tasks is characterized via uniform convergence or PAC-bounds \cite{farid2021generalization, schnitzer2026probabilistic}. On a more empirical side, \cite{mohaya2026transformers} study nonlinear multitask control, demonstrating efficient generalization of nonlinear control policies across different system dynamics, albeit without theoretical generalization guarantees. Our work provides a generalization bound in which the error explicitly depends on the maximum bisimulation-based heterogeneity over the training set.

\vspace{0.2cm}

\noindent \textbf{Notation.} 
We use $\|\cdot\|$ to denote the spectral norm, while $\|\cdot\|_F$ denotes the Frobenius norm. For a matrix $X \in \mathbb{R}^{m \times n}$, we denote by $\operatorname{vec}(X) \in \mathbb{R}^{mn}$ its vectorization obtained by stacking its columns. We use $\ell \in [H] := \{1, \ldots, H\}$ to index elements in a finite collection of size $H$. Let $\texttt{dare}(A,B,Q,R)$ denote the unique stabilizing solution $P$ of the discrete-time algebraic Riccati equation $P = Q + A^\top P A - A^\top P B (R + B^\top P B)^{-1} B^\top P A.$

\section{Setup}\label{sec:problem_formulation}

We consider the linear quadratic Gaussian (LQG) control problem for a collection of discrete-time, linear time-invariant, and partially observed systems, indexed by $i \in \mathbb{N}$, each evolving according to 
\begin{equation}
  \label{eq:mtlqg}
  \begin{aligned}
    x_{t+1}^{(i)} &= A^{(i)} x_t^{(i)} + B^{(i)} u_t^{(i)} + w_t^{(i)}, \\
    y_t^{(i)} &= C^{(i)} x_t^{(i)} + v_t^{(i)}, \; t = 0,1,2,\ldots,
  \end{aligned}
\end{equation}
where $x_t^{(i)} \in \mathbb{R}^{n_x}$ denotes the state, $u_t^{(i)} \in \mathbb{R}^{n_u}$ the control input, and $y_t^{(i)} \in \mathbb{R}^{n_y}$ the  output at time $t$ of system $i$.

The process noise $\{w_t^{(i)}\}_{t \geq 0}$ and measurement noise $\{v_t^{(i)}\}_{t \geq 0}$ are assumed to be independent and identically distributed over time, following zero-mean Gaussian distributions with covariances $W^{(i)} \succeq 0$ and $V^{(i)} \succ 0$, respectively. These noise processes are mutually independent across $i \in \mathbb{N}$ and independent of the initial condition. Without loss of generality, we set $x^{(i)}_0 = 0$ for any $i \in \mathbb{N}$.

\begin{assumption}
\label{ass:ctrl_obs}
The pair $(A^{(i)},B^{(i)})$ is controllable and $(A^{(i)},C^{(i)})$ is observable for all $i \in \mathbb{N}$. 
\end{assumption}

\noindent\textbf{System-specific objective.} The system-specific objective is to design a sequence of control inputs $\{u^{(i)}_t\}_{t=0}^{T-1}$, for a given horizon $T \in \mathbb{N}$, that optimizes the quadratic cost:
\begin{align*}
J^{(i)}:=  \limsup_{T \to \infty} 
  \frac{1}{T}
  \E \left[
    \sum_{t=0}^{T-1}
    \left(
      y_t^{(i)\top} Q^{(i)} y_t^{(i)} 
      + 
      u_t^{(i)\top} R^{(i)} u_t^{(i)}
    \right)
  \right],
\end{align*}
subject to \eqref{eq:mtlqg}, over the set of stabilizing control inputs, where the expectation is taken with respect to the randomness from the process and measurement noises. Here, $Q^{(i)} \succeq 0$ and $R^{(i)} \succ 0$ are the output and input cost matrices, respectively.

\begin{definition} An LQG task indexed by $i \in \mathbb{N}$ is a tuple $\mathcal{T}^{(i)} := (A^{(i)},B^{(i)},C^{(i)}, W^{(i)}, V^{(i)}, Q^{(i)}, R^{(i)})$ that defines the system dynamics \eqref{eq:mtlqg} and cost $J^{(i)}$.
\end{definition}

\noindent \textbf{Task distribution.} We assume that each LQG task $\mathcal{T}^{(i)}$ is drawn independently from an underlying joint distribution $\mathcal{P}_{\mathcal{T}}$ over system and cost parameters. This distribution captures variability in the dynamics, observation models, noise statistics, and objectives across tasks. In general, $\mathcal{P}_{\mathcal{T}}$ is unknown, and we only have access to $N \in \mathbb{N}$ sampled tasks drawn from this distribution.

This reflects practical scenarios where each task corresponds to a system with slightly different physical or operational characteristics. For example, in a cart-pole setting, different tasks may arise from variations in the pole mass, cart mass, and pole length, which induce different system matrices $(A^{(i)}, B^{(i)}, C^{(i)})$. Similarly, variations in sensing or process disturbances lead to different noise covariances $(W^{(i)}, V^{(i)})$, while differing control objectives result in task-dependent cost matrices $(Q^{(i)}, R^{(i)})$. 

\vspace{0.1cm}

\noindent \textbf{Policy.}
When the system and cost parameters $(A^{(i)}, B^{(i)}, C^{(i)}, W^{(i)}, V^{(i)}, Q^{(i)}, R^{(i)})$ are known, the problem of optimizing $J^{(i)}$ is well understood \cite{zhou_robust_1996}. In this setting, the optimal controller follows from the separation principle, which decomposes the problem into state estimation and control. In particular, the state is estimated via a Kalman filter, yielding $\hat{x}^{(i)}_t$, and the control inputs are designed by solving a LQR problem. This yields the policy $u^{(i)}_t :=- K^{(i)} \widehat{x}^{(i)}_t$, with control gain $K^{(i)} \in \mathbb{R}^{n_u \times n_x}$.

In contrast, when the system and cost parameters are unknown, the controller must be learned directly from data. To this end, we adopt a model-free approach based on policy optimization, where the controller is directly optimized using data without explicitly identifying the system dynamics, via policy gradient methods \cite{fazel2018global, fallah2025gradient}. In this work, our goal is to analyze the optimality and generalization properties of policy gradient methods for multitask LQG, and to characterize the impact of task heterogeneity on performance.

\vspace{0.1cm}

\noindent \textbf{Policy gradient for the LQG problem.} Policy gradient methods have been extensively studied for the LQR problem, where global convergence guarantees are established in \cite{fazel2018global} thanks to the Polyak-{\L}ojasiewicz inequality (gradient dominance). However, extending this to the LQG setting is nontrivial due to partial observability and noise. In particular, \cite{mohammadi2021lack} demonstrated the lack of gradient dominance for LQG problems in the space of dynamic output-feedback controllers, which precludes global convergence guarantees.

To address this, \cite{fallah2025gradient}, following \cite{zhao_globally_2023}, introduces an input-output history representation that lifts the partially observed system into a fully observed higher-dimensional state. By constructing a history-dependent state from past inputs and outputs, $\{u^{(i)}_{t-p:t-1}\}_t$ and $\{y^{(i)}_{t-p+1:t}\}_t$, for some $p \in \mathbb{N}$, the LQG problem is recast as an equivalent LQR problem in the lifted space. This is in a similar spirit to \cite{kraisler2024output}, which reparameterizes output-feedback controllers to improve optimization geometry. Here, we build on this representation to analyze policy gradient methods for multitask LQG.

\vspace{0.1cm}
\noindent \textbf{History representation and the lifted controller.} Fix a history length $p \in \mathbb{N}$. For any time $t \geq 0$, define
$$
u_{t,p}^{(i)} :=
\begin{bmatrix}
u_{t-1}^{(i)} \\ \vdots \\ u_{t-p}^{(i)}
\end{bmatrix}, \;
y_{t,p}^{(i)} :=
\begin{bmatrix}
y_t^{(i)} \\ \vdots \\ y_{t-p+1}^{(i)}
\end{bmatrix}, \text{ and }
z_{t,p}^{(i)} :=
\begin{bmatrix}
u_{t,p}^{(i)} \\ y_{t,p}^{(i)}
\end{bmatrix}.
$$
As shown in \cite{fallah2025gradient} (\leo{for a sufficiently large history length $p$}), the estimated state can be expressed as a linear function of the input-output history vector, namely,
\begin{equation}
\label{eq:Sstar-map}
\widehat x_t^{(i)} = \leo{S_p^{(i)}} z_{t,p}^{(i)},
\end{equation}
where the input-output mapping $\leo{S_p^{(i)}}$ has full row rank and is defined in \cite[Section II]{fallah2025gradient}. This mapping depends on the system parameters $(A^{(i)}, B^{(i)}, C^{(i)}, W^{(i)}, V^{(i)})$ as well as the corresponding Kalman filter gain. The sequence of corresponding control inputs can then be written as follows:

\begin{equation}
\label{eq:K-from-Ktilde}
u^{(i)}_t := -K \widehat{x}^{(i)}_t := -K \leo{S_p^{(i)}} z^{(i)}_{t,p}, \; t = 0,1,2,\ldots, 
\end{equation}
which induces a control policy that is linear in the history, that is, $\{u^{(i)}_t\}_t := -\widetilde{K} z^{(i)}_{t,p}$ with $\widetilde{K} := K \leo{S_p^{(i)}}$. Note that, by construction, the dynamics of the input-output history $z^{(i)}_{t,p} \in \mathbb{R}^{p(n_u + n_y)}$ are fully observed. Therefore, policy gradient updates of the form
\begin{align}
\label{eq:PG_iterates_single_task}
\widetilde{K}^{(i)}_{n+1} = \widetilde{K}^{(i)}_{n} - \alpha \nabla J^{(i)}(\widetilde{K}^{(i)}_n),
\end{align}
with $J^{(i)}(\widetilde{K})$ denoting $J^{(i)}$ under $\{u^{(i)}_t\}_{t=0}^{T-1} := -\widetilde{K} z^{(i)}_{t,p}$, for any stabilizing controller $\widetilde{K}$, converge globally to the optimal lifted controller $\widetilde{K}^{(i)}_\star$ for a sufficiently large number of iterations $n$ and an appropriately chosen step size $\alpha > 0$ \cite{fallah2025gradient}. This follows since the lifted problem is equivalent to the minimization of a quadratic cost over a fully observed high-dimensional linear system (i.e., an equivalent LQR problem).

\begin{definition}[Task-specific set of stabilizing controllers]
The optimization of the quadratic cost $J^{(i)}(\widetilde{K})$ subject to \eqref{eq:mtlqg} is carried out over the set of stabilizing lifted controllers 
\begin{equation*}
\widetilde{\mathcal{K}}^{(i)} 
= \left\{ \widetilde{K} \in \mathbb{R}^{n_u \times p(n_u+n_y)} \;\middle|\; 
\rho\left(A^{(i)} - B^{(i)} \widetilde{K} \leo{S^{(i)\dagger}_p}\right) < 1 \right\}.
\end{equation*}
\end{definition}

In multitask LQG, the goal is to learn a single controller that stabilizes all tasks. Since a controller stabilizing one task may fail for another, the feasible set must enforce simultaneous stabilization. This motivates defining a common stabilizing set as the intersection of the task-specific stabilizing sets $\{\widetilde{\mathcal{K}}^{(i)}\}_{i \in \mathbb{N}}$, which we formalize in the sequel.

\vspace{0.1cm}
\noindent \textbf{Multitask LQG problem.} The objective in multitask LQG is to design a lifted controller $\widetilde{K}$ that performs well in expectation over tasks drawn from $\mathcal{P}_{\mathcal{T}}$, namely,
\begin{align}\label{eq:population_cost}
    \widetilde{K}_\star \in \argmin_{\widetilde{K}} \; \E_{\mathcal{T}^{(i)} \sim \mathcal{P}_{\mathcal{T}}} \left[ J^{(i)}(\widetilde{K}) \right] := J(\widetilde{K}),
\end{align}
subject to the system dynamics \eqref{eq:mtlqg}, with $u^{(i)}_{t} := \widetilde{K}z^{(i)}_{t,p}$, for all tasks $\mathcal{T}^{(i)}$ drawn from an unknown distribution $\mathcal{P}_{\mathcal{T}}$.

Let $\mathcal{S} := \{\mathcal{T}^{(i)}\}_{i=1}^N$ denote a training set of $N$ LQG tasks drawn i.i.d. from $\mathcal{P}_{\mathcal{T}}$. We consider the empirical objective
\begin{align}\label{eq:empirical_cost}
    \widetilde{K}_{\star,\mathcal{S}} \in \argmin_{\widetilde{K}} \; \frac{1}{N} \sum_{i=1}^N J^{(i)}(\widetilde{K}) := \hat{J}(\widetilde{K}),
\end{align}
subject to the system dynamics, where the optimization is carried out over a set of controllers that stabilize all tasks in $\mathcal{S}$ and is defined as $\widetilde{\mathcal{K}}_{\mathcal{S}}:= \cap_{\mathcal{T}^{(i)} \in \mathcal{S}} \widetilde{\mathcal{K}}^{(i)}$. To minimize \eqref{eq:empirical_cost} we perform multitask policy gradient iterates of the form
\begin{align}
\label{eq:PG_iterates_multitask}
\widetilde{K}_{n+1} = \widetilde{K}_{n} - \frac{\alpha}{N}\sum_{i=1}^N \nabla J^{(i)}(\widetilde{K}_n),
\end{align}
where each task-specific policy gradient $\nabla J^{(i)}(\widetilde{K}_n)$ will later be estimated from the data.

\vspace{0.1cm}
\begin{center}
\fbox{%
\colorbox{blue!5}{%
\parbox{0.92\linewidth}{%
\textbf{Goal.} Our goal is to demonstrate that, asymptotically as $n \to \infty$, if \eqref{eq:PG_iterates_multitask} converges, the controller denoted by $\widetilde{K}_{\infty,\mathcal{S}}$, is close, in performance, to each task-specific optimal lifted controller $\widetilde{K}^{(i)}_\star$ up to a task heterogeneity-induced bias, and it achieves a small generalization error for any task drawn from $\mathcal{P}_{\mathcal{T}}$, which is itself also governed by task heterogeneity.
}}}
\end{center}

Analyzing \eqref{eq:PG_iterates_multitask} is challenging and has been studied in \cite{toso2024async,fujinami2025policy,stamouli2025policy}. The difficulty arises because the empirical cost $\hat{J}(\widetilde{K})$ does not, in general, satisfy gradient dominance under heterogeneous system and cost parameters. To address this, and following \cite{toso2024async,stamouli2025policy}, we aim to analyze the performance of \eqref{eq:PG_iterates_multitask} by bounding the following task-specific optimality gaps and generalization error. 

\vspace{0.2cm}

\noindent $\bullet$ \textbf{Algorithm-independent gap.} Given $\mathcal{S}$, we upper bound $J^{(i)}(\widetilde{K}_{\star,\mathcal{S}}) - J^{(i)}(\widetilde{K}^{(i)}_\star)$, for all tasks $\mathcal{T}^{(i)} \in \mathcal{S}$.
\vspace{0.2cm}

\noindent $\bullet$ \textbf{Algorithm-dependent gap.} Given $\mathcal{S}$, after $n$ policy gradient iterations following \eqref{eq:PG_iterates_multitask}, we bound $J^{(i)}(\widetilde{K}_{n}) - J^{(i)}(\widetilde{K}^{(i)}_\star)$, for all tasks $\mathcal{T}^{(i)} \in \mathcal{S}$.
\vspace{0.2cm}

\noindent $\bullet$ \textbf{Generalization error.} Given $\mathcal{S}$ and the asymptotic controller $\widetilde{K}_{\infty,\mathcal{S}}$, we bound the gap $|J(\widetilde{K}_{\infty,\mathcal{S}}) - \hat{J}(\widetilde{K}_{\infty,\mathcal{S}})|$.

\vspace{0.1cm}

\begin{definition}[Subset of stabilizing controllers]\label{def:stabilizing_set}
Given the stabilizing controller $\widetilde{K}_0\in \cap_{i=1}^N\widetilde{\calK}^{(i)}$ and constants \leo{$\beta_1,\ldots,\beta_N\geq1$, $\omega_1,\ldots,\omega_N > 0$, $M_1,\ldots,M_N< \infty$}, and \leo{$c_1, \ldots, c_N >0$}, the subset of stabilizing controller is $\widetilde{\calK}_{\mathrm{stab},\mathcal{S}}:=\cap_{i=1}^N\widetilde{\mathcal{K}}^{(i)}_{\mathrm{stab}}(\leo{\beta_i, \omega_i, M_i,c_i})$,  with every $\widetilde{K} \in \widetilde{\mathcal{K}}^{(i)}_{\mathrm{stab}}(\leo{\beta_i, \omega_i, M_i,c_i})$ satisfying
\begin{align*}
&J^{(i)}(\widetilde{K})\hspace{-0.05cm}-\hspace{-0.05cm}J^{(i)}(\widetilde{K}_\star^{(i)})\nonumber
\hspace{-0.05cm}\leq\hspace{-0.05cm}\beta_i(J^{(i)}(\widetilde{K}_0)\hspace{-0.05cm}-\hspace{-0.05cm}J^{(i)}(\widetilde{K}_\star^{(i)})),\\ 
&\leo{\rho\left(A^{(i)} - B^{(i)} \widetilde{K}\leo{S^{(i)\dagger}_p}\right) \leq 1-\omega_i,}\\
&\leo{\|\widetilde{K}\|_F \leq M_i,}\\
&\leo{\sigma_{\min}\left(\mathcal{C}(A^{(i)}_{\widetilde{K}})\right) \geq c_i},
\end{align*}
where $\mathcal{C}({A}^{(i)}_{\widetilde{K}}) := \left[ \Sigma^{(i)1/2}_\nu \; {A}^{(i)}_{\widetilde{K}}\Sigma^{(i)1/2}_\nu \; \ldots \; {A}^{(i)n_x - 1}_{\widetilde{K}}\Sigma^{(i)1/2}_\nu \right]$, and  ${A}^{(i)}_{\widetilde{K}} :=  A^{(i)} - B^{(i)}\widetilde{K} S^{(i)\dagger}_p$.
\end{definition}

We also collect here the properties of the lifted task-specific cost provided in \cite{fallah2025gradient}. For any task $\mathcal{T}^{(i)} \in \mathcal{S}$ and lifted stabilizing controllers $\widetilde{K}, \widetilde{K}^\prime \in \widetilde{\mathcal{K}}^{(i)}_{\mathrm{stab}}(\leo{\beta_i, \omega_i, M_i,c_i})$, the gradient is $L_i$-smooth, i.e.,
$$
\|\nabla J^{(i)}(\widetilde{K}) - \nabla J^{(i)}(\widetilde{K}^\prime)\|
\leq L_i \|\widetilde{K} - \widetilde{K}^\prime\|,
$$
and the cost and gradient are uniformly bounded as $J^{(i)}(\widetilde{K}) \le \bar{J}_i < \infty$, \leo{$\|\nabla J^{(i)}(\widetilde{K})\|_F \le G_i < \infty$, respectively}. Moreover, the cost satisfies
\begin{align}\label{gradient_dominance}
J^{(i)}(\widetilde{K}) - J^{(i)}(\widetilde{K}^{(i)}_\star)
\leq \frac{1}{\gamma_i(\widetilde{K})} \|\nabla J^{(i)}(\widetilde{K})\|_F^2,
\end{align}\\[-0.3cm]
\leo{where $\gamma_i(\widetilde{K})$ is a positive coefficient dependent on the controller $\widetilde{K}$.}


These properties follow the results for LQR \cite{fazel2018global,gravell2020learning} and extend to our setting by expressing the controller in the lifted input-output representation, namely, $K = \widetilde{K} \leo{S^{(i)\dagger}_p}$. \leo{The main distinction from the LQR setting, as highlighted in \cite{fallah2025gradient}, is that the gradient-dominance coefficient $\gamma_i(\widetilde{K})$ depends on the controller $\widetilde{K}$.}
For the multitask setting, policy gradient updates are computed by aggregating task-specific gradients \eqref{eq:PG_iterates_multitask}. While this aggregation reduces estimation variance (in the model-free setting), it also introduces a fundamental challenge: the gradients associated with different tasks need not align. This misalignment arises from heterogeneity in system dynamics, noise statistics, and cost parameters.

In particular, even when each task-specific objective satisfies favorable properties such as gradient dominance, these properties do not generally transfer to the aggregated objective due to discrepancies across task gradients. As a result, the performance of multitask policy gradient methods is fundamentally governed by the degree of gradient heterogeneity across tasks in the training set $\mathcal{S}$.

\begin{definition}[Gradient heterogeneity]
For a common stabilizing controller $\widetilde{K} \in \widetilde{\mathcal{K}}_{\mathrm{stab},\mathcal{S}}$, and any two tasks $\mathcal{T}^{(i)}, \mathcal{T}^{(j)} \in \mathcal{S}$, the gradient heterogeneity is defined as
\begin{align}\label{eq:gradient_heterogeneity}
\epsilon^{(ij)}_{\mathrm{het}}\left(\widetilde{K}\right)
    := \|\nabla J^{(i)}(\widetilde{K}) - \nabla J^{(j)}(\widetilde{K})\|_F^2.
\end{align}
\end{definition}
Note that gradient heterogeneity directly impacts the optimization error in multitask policy gradient. In particular, evaluating the task-specific optimality gap using the iterate $\widetilde{K}_n$ generated by \eqref{eq:PG_iterates_multitask} introduces an error term of order $\epsilon^{(ij)}_{\mathrm{het}}\left(\widetilde{K}_n\right)$ at each iteration. Therefore, characterizing this quantity is crucial, as it captures the bias induced by task heterogeneity and limits how close the learned common controller can approach each task-specific optimum.

\section{The Gradient Dynamics}
\label{subsec:lyap-cost}

To quantify and control gradient heterogeneity, we next provide a dynamical systems interpretation of the task-specific policy gradients. This allows for the characterization of gradient heterogeneity via approximate bisimulation \cite{stamouli2025policy}.

Define $\widetilde Q^{(i)} := C^{(i)\top} Q^{(i)} C^{(i)}$, and for each task let $\Sigma_\nu^{(i)}$ be the estimation error covariance: $$\Sigma_\nu^{(i)} := L^{(i)}\big(C^{(i)}\widetilde{\Sigma}^{(i)}C^{(i)\top} + V^{(i)}\big)L^{(i)\top},$$
with $ L^{(i)}:= \widetilde{\Sigma}^{(i)} C^{(i)\top}(C^{(i)}\widetilde{\Sigma}^{(i)}C^{(i)\top}+V^{(i)})^{-1}$, and covariance $\widetilde{\Sigma}^{(i)}:= \texttt{dare}(A^{(i)\top},C^{(i)\top},W^{(i)},V^{(i)})$. Importantly, $\Sigma_\nu^{(i)}$ is independent of $\widetilde K$ (i.e., by the separation principle).

For any $\widetilde K \in \widetilde{\mathcal{K}}_{\mathrm{stab},\mathcal{S}}$, let
$P_{\widetilde K}^{(i)}$ and $\Sigma_{\widetilde K}^{(i)}$ be the unique solution of the discrete-time Lyapunov equations
\begin{align*}
P_{\widetilde K}^{(i)}
&=\widetilde Q^{(i)} + S_p^{(i)\dagger \top} \widetilde K^\top R^{(i)} \widetilde K\, S_p^{(i)\dagger}
+ A_{\widetilde K}^{(i) \top} P_{\widetilde K}^{(i)} A_{\widetilde K}^{(i)},
\\
\Sigma_{\widetilde K}^{(i)}
&=\Sigma_\nu^{(i)} + A_{\widetilde K}^{(i)} \Sigma_{\widetilde K}^{(i)} A_{\widetilde K}^{(i) \top},
\end{align*}
with $A^{(i)}_{\widetilde{K}}:= A^{(i)} - B^{(i)}\widetilde{K}\leo{S^{(i)\dagger}_p}$, for all $\mathcal{T}^{(i)}$ in $\mathcal{S}$. We define
\begin{equation*}
E_{\widetilde K}^{(i)} :=
2\left(\Big(R^{(i)} \hspace{-0.1cm}+\hspace{-0.1cm} B^{(i)\top} P_{\widetilde K}^{(i)} B^{(i)}\Big)\, \widetilde K S_p^{(i)\dagger}
\hspace{-0.1cm}-\hspace{-0.1cm} B^{(i) \top} P_{\widetilde K}^{(i)} A^{(i)}\right).
\end{equation*}

\leo{With the above quantities defined, we introduce the gradient dominance coefficient from \eqref{gradient_dominance} as follows:
$$\gamma_i(\widetilde{K}) := 4 \leo{\lambda_{\min}(\Sigma_{\widetilde{K}}^{(i)})}^2 \lambda_{\min}(R^{(i)})/(\|\Sigma_{\widetilde{K}^{(i)}_\star}^{(i)}\| \|\leo{S_p^{(i)}}\|^2).$$}
\vspace{-0.5cm}
\begin{definition}(Closed-form policy gradient) Given a stabilizing controller $\widetilde{K} \in \widetilde{\mathcal{K}}^{(i)}$ and a task $\mathcal{T}^{(i)} \in \mathcal{S}$, the task-specific gradient at $\widetilde{K}$ is $ \nabla J^{(i)}(\widetilde{K}) := E_{\widetilde K}^{(i)} \Sigma_{\widetilde K}^{(i)}  S_p^{(i)\dagger \top}.$
\end{definition}
Moreover, for each task and stabilizing $\widetilde K \in \widetilde{\mathcal{K}}^{(i)}_{\mathrm{stab}}$, we define the covariance recursion as follows:
\begin{equation}
\label{eq:Sigma-recursion}
\Sigma^{(i)}_{\widetilde K,t+1}
=A_{\widetilde K}^{(i)} \Sigma^{(i)}_{\widetilde K,t} A_{\widetilde K}^{(i)\top}
+ \Sigma_\nu^{(i)},
\end{equation}
which implies that $\Sigma^{(i)}_{\widetilde K} := \lim_{T\to \infty }\frac{1}{T}\sum_{t=0}^{T-1}\Sigma^{(i)}_{\widetilde{K},t}$.

Note that for every task and any stabilizing controller $\widetilde{K} \in \widetilde{\mathcal{K}}^{(i)}_{\text{stab}}$, we can characterize the dynamics of the vectorized gradient for every time $t$ through the following system
\begin{align}\label{eq:matrix_system}
    S_{\widetilde K}^{(i)}:\left\{
                \begin{array}{lcr}
                    s_{\widetilde{K},t+1}^{(i)}=F_{\widetilde{K}}^{(i)} s_{\widetilde{K},t}^{(i)}+\nu^{(i)}\\
                    z_{\widetilde{K},t}^{(i)}=C_{\widetilde{K}}^{(i)} s_{\widetilde{K},t}^{(i)}
                \end{array}
                \right.,
\end{align}
where $s_{\widetilde{K},t}^{(i)}:= \operatorname{vec}(\Sigma^{(i)}_{\widetilde{K},t})\in \mathbb{R}^{n^2_x},$ $F^{(i)}_{\widetilde{K}} := A^{(i)}_{\widetilde{K}} \otimes A^{(i)}_{\widetilde{K}}$, $C^{(i)}_{\widetilde{K}}:= \leo{S^{(i)\dagger}_p} \otimes E^{(i)}_{\widetilde{K}}$, and $\nu^{(i)}:= \operatorname{vec}(\Sigma^{(i)}_\nu)$.

Note that the outputs of $S^{(i)}_{\widetilde{K}}$ and $S^{(j)}_{\widetilde{K}}$, for any pair of distinct tasks under a common stabilizing controller $\widetilde{K}$, converge to the gradient heterogeneity as $t \to \infty$, i.e.,
\begin{align}\label{eq:convergence_covariance_system}
   \epsilon^{(ij)}_{\mathrm{het}}\left(\widetilde{K}\right)= \lim_{T\to \infty} \frac{1}{T}\sum_{t=0}^{T-1}\|z^{(i)}_{\widetilde{K},t} - z^{(j)}_{\widetilde{K},t}\|^2,
\end{align}

We emphasize that \eqref{eq:convergence_covariance_system} follows from Cesàro's mean theorem \cite[Chapter 5]{hardy2024divergent} 
applied to the sequence of output discrepancies between the systems $S^{(i)}_{\widetilde{K}}$ and $S^{(j)}_{\widetilde{K}}$. This sequence is convergent as $A^{(i)}_{\widetilde{K}}$ and $A^{(j)}_{\widetilde{K}}$ are Schur stable.

We emphasize that the dynamical system  \eqref{eq:matrix_system} differs from the Lyapunov matrix system used in the multitask LQR analysis of \cite{stamouli2025policy}. For LQR, the gradient dynamics can be characterized directly through the state covariance $\Sigma^{(i)}_K$, which satisfies a standard discrete Lyapunov recursion.
In contrast, for LQG, the gradient takes the form $\nabla J^{(i)}(\widetilde K) = E^{(i)}_{\widetilde K}\Sigma^{(i)}_{\widetilde K}S^{(i)\dagger\top}_p$, which depends not only on the closed-loop covariance but also on the lifting operator $\leo{S^{(i)\dagger}_p}$ that relates the input–output history to the state estimate. As a result, the gradient dynamics cannot be described only through the covariance recursion. Instead, we consider the lifted vectorized state $s^{(i)}_{\widetilde K,t}=\mathrm{vec}(\Sigma^{(i)}_{\widetilde K,t})$, whose evolution is governed by the linear system $s^{(i)}_{\widetilde K,t+1}=F^{(i)}_{\widetilde K}s^{(i)}_{\widetilde K,t}+\nu^{(i)}$ with output $z^{(i)}_{\widetilde K,t}=C^{(i)}_{\widetilde K}s^{(i)}_{\widetilde K,t}$.

The gradient dynamics provide a dynamical systems interpretation of task-specific gradients, reducing their comparison to analyzing output mismatches between induced gradient systems. To this end, approximate bisimulation \cite{girard2011approximate,abate2013approximation,stamouli2025policy,stamouli2025layered} offers a principled way to certify behavioral similarity. We leverage bisimulation functions to characterize gradient heterogeneity via these output mismatches.

\section{Bisimulation Functions for Multitask LQG}

In the following, we consider a common stabilizing controller $\widetilde{K} \in \widetilde{\mathcal{K}}_{\mathrm{stab},\mathcal{S}}$. We begin by characterizing the conditions under which a Lyapunov function induces a bisimulation function between the two closed-loop systems $S_{\widetilde{K}}^{(i)}$ and $S_{\widetilde{K}}^{(j)}$.

\begin{definition}[Bisimulation Function]\label{def:bisim_function}
A continuous function $V: \mathbb{R}^{n^2_x}\times \mathbb{R}^{n^2_x}\to\setR_+$ is a bisimulation function between $S_{\widetilde{K}}^{(i)}$ and $S_{\widetilde{K}}^{(j)}$ if there exists $\lambda^\prime \in (0,2), \zeta > 0$ such that 
\begin{subequations}
\begin{align}
\label{eq:BF_condition_1}
    V&(s_{\widetilde{K},t}^{(i)},s_{\widetilde{K},t}^{(j)})\geq \left\|C_{\widetilde K}^{(i)} s_{\widetilde{K},t}^{(i)}-C_{\widetilde{K}}^{(j)} s_{\widetilde{K},t}^{(j)}\right\|^2, \text{ and }
\end{align}
\begin{align}
    V(s_{\widetilde{K},t+1}^{(i)},&s_{\widetilde{K},t+1}^{(j)})-V(s_{\widetilde{K},t}^{(i)},s_{\widetilde{K},t}^{(j)})\nonumber\\
    \label{eq:BF_condition_2}
    &\leq -\lambda^\prime V(s_{\widetilde{K},t}^{(i)},s_{\widetilde{K},t}^{(j)})+\zeta V(\nu^{(i)},\nu^{(j)}).
\end{align}
for all pair of states $(s_{\widetilde{K},t}^{(i)},s_{\widetilde{K},t}^{(j)})\in\mathbb{R}^{n^2_x}\times \mathbb{R}^{n^2_x}$.
\end{subequations}
\end{definition}

The first condition \eqref{eq:BF_condition_1} ensures that the output mismatch between $S_{\widetilde{K}}^{(i)}$ and $S_{\widetilde{K}}^{(j)}$, for any pair of states $(s_{\widetilde{K},t}^{(i)}, s_{\widetilde{K},t}^{(j)}) \in \mathbb{R}^{n_x^2} \times \mathbb{R}^{n_x^2}$, is upper bounded by the bisimulation function $V(s_{\widetilde{K},t}^{(i)}, s_{\widetilde{K},t}^{(j)})$.  On the other hand,  condition \eqref{eq:BF_condition_2}, with $\lambda^\prime \in (0,2)$, enforces an exponential decay of the bisimulation function along the system trajectories. In particular, it implies that asymptotically $V(s_{\widetilde{K},t}^{(i)}, s_{\widetilde{K},t}^{(j)})$ is bounded by $V(\nu^{(i)}, \nu^{(j)}) / \lambda^\prime$ as stated in the following lemma.

\begin{lemma} \label{lemma:bisim_function} Let $V:\mathbb{R}^{n^2_x}\times \mathbb{R}^{n^2_x}\to\setR_+$ be a bisimulation function between $S_{\widetilde{K}}^{(i)}$ and $S_{\widetilde{K}}^{(j)}$. It holds that
\begin{equation}\label{eq:asympt_bis_function_bound}
    \hspace{-0.3cm}\lim_{T\to\infty}\frac{1}{T}\sum_{t=0}^{T-1}\|z_{\widetilde K,t}^{(i)}-z_{\widetilde K,t}^{(j)}\|^2\leq \zeta V(\nu^{(i)},\nu^{(j)})/\lambda^\prime, \text{ and }
\end{equation}
\begin{equation}\label{eq:grad_het_V_bound}
   \epsilon^{(ij)}_{\mathrm{het}}\left(\widetilde{K}\right) \leq \zeta V(\nu^{(i)},\nu^{(j)})/\lambda^\prime. 
\end{equation}
\end{lemma}

\begin{proof}
The result follows by recursively applying \eqref{eq:BF_condition_2}, which yields a geometric decay of the bisimulation function up to a steady-state bound determined by $\zeta V(\nu^{(i)},\nu^{(j)})/\lambda^\prime$. Combining this with \eqref{eq:BF_condition_1}, which upper bounds the output mismatch, and taking Ces\`aro means \cite{hardy2024divergent} leads to the desired result. Full details are provided in the appendix.
\end{proof}

We now specify the bisimulation function used to characterize gradient heterogeneity. To this end, we establish the existence of a PSD matrix that parameterizes such function.

\begin{lemma}\label{lemma:existence_M} Given a common stabilizing controller $\widetilde{K} \in \widetilde{\mathcal{K}}_{\mathcal{S}}$ and any pair of distinct LQG tasks $(\mathcal{T}^{(i)},\mathcal{T}^{(j)})$ from the training set $\mathcal{S}$. Then, there exists $M \in \mathbb{S}^{2n^2_x}$ and $\lambda \in (0,1)$ such that the following hold
\begin{subequations}
\begin{align}
    \label{eq:M_condition}
    &M\succeq C_{\widetilde{K}}^{(ij)\intercal}C_{\widetilde{K}}^{(ij)}, \\
    \label{eq:constraint_stabilization}
    &F_{\widetilde{K}}^{(ij)\intercal}M F_{\widetilde K}^{(ij)}-M\preceq-\lambda M,
\end{align}
\end{subequations}
with $C_{\widetilde{K}}^{(ij)}:=\left[C_{\widetilde{K}}^{(i)} \; -C_{\widetilde{K}}^{(j)}\right]$ and $F_{\widetilde{K}}^{(ij)}:=\mydiag(F_{\widetilde{K}}^{(i)},F_{\widetilde{K}}^{(j)})$.
\end{lemma}

Since $F^{(ij)}_{\widetilde{K}}$ is Schur stable for any stabilizing controller $\widetilde{K} \in \widetilde{\mathcal{K}}_{\mathcal{S}}$, the conditions \eqref{eq:M_condition} and \eqref{eq:constraint_stabilization} follow directly (see Appendix \ref{app:bisimulation}). We then use $M \in \mathbb{S}^{2 n^2_x}_+$ to parameterize a quadratic function, which will serve as our bisimulation function.

\begin{theorem}\label{theorem:Bisimulation_Functions_characterization}
Given a common stabilizing controller $\widetilde{K} \in \widetilde{\mathcal{K}}_{\mathcal{S}}$ and any pair of distinct LQG tasks $(\mathcal{T}^{(i)},\mathcal{T}^{(j)})$ from the training set $\mathcal{S}$. Let $M\in \setS_+^{2 n^2_x}$ and $\lambda\in(0,1)$ satisfy \eqref{eq:M_condition} and \eqref{eq:constraint_stabilization}. Moreover, fix $\eta \in \left(0,\frac{\lambda}{1-\lambda}\right)$. Then, the quadratic function $V:\mathbb{R}^{n^2_x}\times \mathbb{R}^{n^2_x}\to\setR_+$ given by
\begin{equation}\label{eq:bisimulation_function}
V(s_{\widetilde{K},t}^{(i)},s_{\widetilde{K},t}^{(j)}) := s^{(ij)\top}_{\widetilde{K},t} M s^{(ij)}_{\widetilde{K},t},
\end{equation}\\[-0.3cm]
with $s^{(ij)}_{\widetilde{K},t} := \left[s^{(i)\top}_{\widetilde{K},t} \; s^{(j)\top}_{\widetilde{K},t}\right]^\top$, is a bisimulation function between systems $S_{\widetilde{K}}^{(i)}$ and $S_{\widetilde{K}}^{(j)}$, implying
\begin{equation}\label{eq:grad_het_bis_bound}
\epsilon^{(ij)}_{\mathrm{het}}\left(\widetilde{K}\right)  \leq \frac{ \zeta\nu^{(ij)\top} M \nu^{(ij)}}{\lambda^\prime}, \text{ with } \nu^{(ij)} =  \begin{bmatrix}
        \nu^{(i)}\\ \nu^{(j)}
    \end{bmatrix},
\end{equation}\\[-0.3cm]
where $\zeta:= 1+\eta^{-1}$ and $\lambda^\prime:= \lambda -\eta(1-\lambda)$.
\end{theorem}

\begin{proof} The result follows by verifying the bisimulation conditions \eqref{eq:BF_condition_1} and \eqref{eq:BF_condition_2} for the gradient dynamics. While similar arguments appear in the bisimulation literature \cite{girard2011approximate,lavaei2017compositional,stamouli2025policy}, our proof is tailored to the lifted multitask LQG problem with discrete-time dynamics. Full details are provided in Appendix \ref{app:bisimulation}.
\end{proof}

\begin{definition}[Bisimulation-based LQG task heterogeneity]
Given a common stabilizing controller $\widetilde{K} \in \widetilde{\mathcal{K}}_{\mathrm{stab},\mathcal{S}}$ and any pair of distinct LQG tasks $(\mathcal{T}^{(i)}, \mathcal{T}^{(j)}) \in \mathcal{S}$, let $M^{(ij)}_{\widetilde{K}}$ and $\lambda^{(ij)}_{\widetilde{K}}$ denote the minimizers of \eqref{eq:grad_het_bis_bound} subject to \eqref{eq:M_condition} and \eqref{eq:constraint_stabilization}. The bisimulation-based heterogeneity between $\mathcal{T}^{(i)}$ and $\mathcal{T}^{(j)}$ under the controller $\widetilde{K}$ is defined as
\begin{equation}\label{eq:bisim_het}
    b_{ij}(\widetilde{K}) := \frac{\left(1+\eta^{(ij)-1}_{\widetilde{K}}\right)\nu^{(ij)\top} M^{(ij)}_{\widetilde{K}} \nu^{(ij)}}{\lambda^{(ij)}_{\widetilde{K}} - \eta^{(ij)}_{\widetilde{K}}\left(1-\lambda^{(ij)}_{\widetilde{K}}\right)},
\end{equation}
where $\eta^{(ij)}_{\widetilde{K}} := \frac{1}{\sqrt{1-\lambda^{(ij)}_{\widetilde{K}}}} - 1$. 
\end{definition} 

The quantity $b_{ij}(\widetilde{K})$ captures the cost gradient mismatch between tasks $\mathcal{T}^{(i)}$ and $\mathcal{T}^{(j)}$ through the steady-state effect of their noise inputs $\nu^{(ij)}$, weighted by the matrix $M^{(ij)}_{\widetilde{K}}$. The matrix $M^{(ij)}_{\widetilde{K}}$ characterizes how differences in the induced gradient dynamics propagate along the system trajectories and is obtained by solving a convex optimization problem (see the appendix). The scalar $\lambda^{(ij)}_{\widetilde{K}} \in (0,1)$ captures the contraction rate of the joint system and is directly related to the stability margin of $F^{(ij)}_{\widetilde{K}}$. Therefore, smaller values of $b_{ij}(\widetilde{K})$ indicate that the tasks are more similar in terms of their gradient dynamics.

Note that our quadratic bisimulation function is defined on the lifted gradient dynamics as in \eqref{eq:bisimulation_function}, capturing mismatch in the closed-loop covariance evolution. In contrast, relevant to this work, is \cite{lavaei2017compositional} that measures probabilistic output closeness between an abstract and a concrete stochastic system. Thus, while both are quadratic, ours captures gradient heterogeneity via lifted dynamics, whereas theirs quantifies abstraction error at the state-trajectory level.

\begin{remark}
It is important to emphasize that, compared to the multitask LQR setting studied in \cite{stamouli2025policy}, the LQG formulation with input-output history representation induces a gradient dynamical system whose dimension scales as $n_x p (n_u + n_y)$. Moreover, the gradient expression depends jointly on the state covariance dynamics and the history representation matrix, which couples estimation and control effects. As a result, the system matrices inherently capture both control and estimation components.

In contrast to \cite{stamouli2025policy}, the bisimulation-based heterogeneity measure in \eqref{eq:bisim_het} does not depend on $\sqrt{\lambda_{\min}(M^{(ij)}_{\widetilde{K}})}$ in the denominator. This ensures that $b_{ij}(\widetilde{K})$ appropriately vanishes when the tasks are identical at optimum $\widetilde{K}_\star := \widetilde{K}^{(i)}_\star = \widetilde{K}^{(j)}_\star$.
\end{remark}

Finally, for each task, define the average bisimulation-based heterogeneity as $ b_i(\widetilde{K}) := \frac{1}{N-1} \sum_{\substack{j=1 \\ j \neq i}}^N b_{ij}(\widetilde{K}),$
which measures the average heterogeneity between task $\mathcal{T}^{(i)}$ and the remaining tasks in the training set $\mathcal{S}$.

\section{Policy Gradient Bounds}\label{sec:policy_gradient_bounds}
We now establish policy gradient bounds for multitask LQG, quantifying the impact of task heterogeneity on both the empirical optimizer $\widetilde{K}_{\star,\mathcal{S}}$ and the policy gradient iterates $\widetilde{K}_n$ in \eqref{eq:PG_iterates_multitask}, given a training set $\mathcal{S}$.

\leo{\begin{assumption}\label{assump:stabilizing_controller} $\widetilde{K}_0, \widetilde{K}_{\star,\mathcal{S}}, \widetilde{K}^{(i)}_\star \in \widetilde{\mathcal{K}}_{\mathrm{stab},\mathcal{S}}$ for all $\mathcal{T}^{(i)} \in \mathcal{S}$.
\end{assumption}}

This assumption is standard in policy gradient control, ensuring finite costs and gradients \cite{fazel2018global,gravell2020learning,fallah2025gradient}. 
\leo{Moreover, note that Assumption \ref{assump:stabilizing_controller} also ensures, $\left(A^{(i)} - B^{(i)}\widetilde{K}_0 S^{(i)\dagger}_p,\Sigma^{(i)1/2}_\nu\right)$ is controllable for all tasks $\mathcal{T}^{(i)} \in \mathcal{S},$ i.e., under the common initial lifted controller $\widetilde{K}_0$, the stochastic input associated with $\Sigma^{(i)}_\nu$ excites every state-space direction of the closed-loop dynamical system associated with each task $\mathcal{T}^{(i)} \in \mathcal{S}$. We emphasize that Assumption \ref{assump:stabilizing_controller} guarantees that $\lambda_{\min}(\Sigma_{\widetilde{K}_0})>0$. We later demonstrate that, for a sufficiently small step size, $\lambda_{\min}(\Sigma_{\widetilde{K}_n})>0$ is preserved throughout the iterations.}

\begin{theorem}[Algorithm-independent Optimality Bound] \label{thm:algor_indep_bounds}
Let $\widetilde{K}_{\star,\mathcal{S}}\in \widetilde{\calK}_{\text{stab},\mathcal{S}}$ be the multitask empirical minimizer for \eqref{eq:empirical_cost}. Then, for every task $\mathcal{T}^{(i)}$ in the training set $\mathcal{S}$, the task-specific optimality gap associated with  $\widetilde{K}_{\star,\mathcal{S}}$ satisfies
\begin{equation}\label{eq:gap_bound_K_star}
    J^{(i)}(\widetilde{K}_{\star,\mathcal{S}})-J^{(i)}(\widetilde{K}_\star^{(i)})\leq \frac{\big\|\leo{S_p^{(i)}}\big\|^2\big\|\Sigma_{\widetilde{K}_\star^{(i)}}^{(i)}\big\|b_i(\widetilde{K}_{\star,\mathcal{S}})}{4\leo{\lambda_{\min}(\Sigma_{\widetilde{K}_{\star,\mathcal{S}}}})^2 \lambda_{\min}(R^{(i)})}.
\end{equation}
\end{theorem}
\begin{proof} This result follows directly from \eqref{gradient_dominance}, along with the bound for $\epsilon^{(ij)}_{\mathrm{het}}\left(\widetilde{K}_{\star,\mathcal{S}}\right)$ in terms of the bisimulation-based heterogeneity measure $b_{ij}(\widetilde{K}_{\star,\mathcal{S}})$.
\end{proof}


\begin{theorem}[Algorithm-dependent Optimality Bound] \label{theorem:algorithm_dependent} Given an initialization $\widetilde{K}_0 \in \widetilde{\calK}_{\mathrm{stab},\mathcal{S}}$, consider the policy gradient iterates $\{\widetilde{K}_n\}_{n=0}^\infty$. Let $\beta_1,\ldots,\beta_N \geq 1$, $\omega_1,\ldots,\omega_N > 0$, and $M_1, \ldots, M_N < \infty$, parameterize the set $\widetilde{\mathcal{K}}_{\mathrm{stab},\mathcal{S}}$. Suppose that the step-size satisfies the conditions of Lemma \ref{lem:policy_gradient_stability} and that \leo{$$\sup_{\widetilde{K} \in \widetilde{\calK}_{\mathrm{stab},\mathcal{S}}} b_i(\widetilde{K}) \leq \inf_{\widetilde{K} \in \widetilde{\calK}_{\mathrm{stab},\mathcal{S}}}\frac{\gamma_i(\widetilde{K})}{6}\big(J^{(i)}(\widetilde{K}_0) - J^{(i)}(\widetilde{K}_\star^{(i)})\big).$$}
Then, $\widetilde{K}_n \in \widetilde{\calK}_{\mathrm{stab},\mathcal{S}}$ for all $n \in \mathbb{N}$. Moreover, for each task in the training set, the asymptotic optimality gap satisfies
\begin{align}\label{eq:gap_bound_K_infty} \limsup_{n\to\infty}&(J^{(i)}(\widetilde{K}_n)-J^{(i)}(\widetilde{K}_\star^{(i)}))\nonumber\\[-0.2cm] &\leq \limsup_{n\to\infty}\frac{3 \big\|\leo{S_p^{(i)}}\big\|^2\big\|\Sigma_{\widetilde{K}_\star^{(i)}}^{(i)}\big\| b_i(\widetilde{K}_n)}{4\underline{\lambda_{\min}}_\Sigma  \lambda_{\min}(R^{(i)})},\end{align}
where $\underline{\lambda_{\min}}_\Sigma := \leo{\inf_{\widetilde{K} \in \widetilde{\mathcal{K}}_{\mathrm{stab},\mathcal{S}}}\lambda_{\min}(\Sigma_{\widetilde{K}}^{(i)})}^2$.
\end{theorem}

Therefore, if the policy gradient iterates converge to $\widetilde{K}_{\infty,\mathcal{S}}$ for a given training  set $\mathcal{S}$, then for all tasks in that set,  inequality \eqref{eq:gap_bound_K_infty} implies
\begin{align*}
   J^{(i)}(\widetilde{K}_{\infty,\mathcal{S}})-J^{(i)}(\widetilde K_\star^{(i)})\leq \frac{3 \big\|\leo{S_p^{(i)}}\big\|^2\big\|\Sigma_{\widetilde{K}_\star^{(i)}}^{(i)}\big\| b_i(\widetilde{K}_{\infty,\mathcal{S}})}{4\underline{\lambda_{\min}}_\Sigma \lambda_{\min}(R^{(i)})}.
\end{align*}

\begin{proof} The proof of Theorem \ref{theorem:algorithm_dependent} follows from a similar stability and convergence analysis as in \cite{stamouli2025policy} for the multitask LQR problem. The main differences arise from the lifted controller parameterization, which introduces additional dependence on the history representation matrix in the bounds. We defer the full proof to Appendix \ref{app:pg_bounds}.
\end{proof}

Theorems \ref{thm:algor_indep_bounds} and \ref{theorem:algorithm_dependent} demonstrate that the performance of a common controller in multitask LQG is fundamentally limited by task heterogeneity. The algorithm-independent bound quantifies the inherent bias induced by optimizing a common controller across heterogeneous systems, while the algorithm-dependent bound shows that policy gradient iterates converge to a neighborhood of the task-specific optima whose size is controlled by the bisimulation-based heterogeneity $b_i(\widetilde{K})$. In particular, when $b_i(\widetilde{K})$ are small, both the empirical optimizer and the policy gradient iterates achieve near-optimal performance for all tasks.

\vspace{0.1cm}
\section{Generalization Guarantees}
\label{sec:generalization}
\vspace{0.1cm}

We now establish generalization guarantees for multitask LQG, characterizing how the performance of the learned controller on \emph{unseen tasks} depends on the training set size and task heterogeneity. Before presenting our guarantees, we first define the quantities: $J_{\star, \mathcal{S}} := \max_{\mathcal{T}^{(i)} \in \mathcal{S}} J^{(i)}(\widetilde{K}^{(i)}_\star)$, $b_{\mathcal{S}}: = \max_{\mathcal{T}^{(i)} \in \mathcal{S}} b_i(\widetilde{K}_{\infty,\mathcal{S}})$, and 
$
\mu_{\mathcal{S}}:= \max_{\mathcal{T}^{(i)} \in \mathcal{S}} \sup_{\widetilde{K}\in \widetilde{\calK}_{\mathrm{stab},\mathcal{S}}} {3/\gamma_i(\widetilde{K})}.
$

\begin{assumption}\label{assump:uniform_stabilization} We assume that, with probability at least $1 - \delta$, the asymptotically learned controller $\widetilde{K}_{\infty,\mathcal{S}}$ stabilizes all tasks in the support $\operatorname{supp}(\mathcal{P}_{\mathcal{T}})$, for some $\delta\in(0,1)$.
\end{assumption}
This assumption is necessary to ensure that the population cost $J(\widetilde{K}_{\infty,\mathcal{S}})$ is well-defined, as instability leads to unbounded costs. In practice, this condition is reasonable when the training set provides sufficient coverage of the task distribution, so that the learned controller generalizes its stabilizing properties beyond the training tasks. In \cite{fujinami2025policy}, this is achieved under a restrictive condition on the difference across task-specific system parameters. Understanding whether such restriction is fundamental or can be relaxed remains an important direction for future work.

\begin{theorem}[Generalization Bound]  \label{theorem:generalization} Given a training set $\mathcal{S}$, an independent validation set $\mathcal{S}_{\text{val}}$ with size $N_{\text{val}}$, the controller $\widetilde{K}_{\infty,\mathcal{S}}$ satisfying Assumption \ref{assumption:uniform out=of-sample cost bound}, and a fixed probability $\delta^\prime \in (0,1)$. Then, with probability $1-\delta^\prime-\delta$, the generalization error under the asymptotically learned controller $\widetilde{K}_{\infty,\mathcal{S}}$ is bounded as
\begin{align}\label{generalization_error}
 \hspace{-0.3cm}\left|J(\widetilde{K}_{\infty,\mathcal{S}}) - \hat{J}(\widetilde{K}_{\infty,\mathcal{S}}) \right|  \leq (J_{\star,\mathcal{S}} + \mu_{\mathcal{S}} b_{\mathcal{S}})\sqrt{\frac{\log(4/(\delta^\prime+\delta))}{2N_{\text{val}}}},
\end{align}
where $\hat{J}(\cdot)$ is computed with respect to the validation set.
\end{theorem}
\begin{proof} The proof follows by combining Hoeffding's inequality \cite{vershynin2018high} with the algorithm-dependent bounds in Theorem \ref{theorem:algorithm_dependent}. We defer the full details to the appendix.
\end{proof}

\begin{remark}[The effect of heterogeneity in generalization] The generalization bound in \eqref{generalization_error} explicitly depends on the bisimulation-based heterogeneity term $b_{\mathcal{S}}$, revealing how task mismatch impacts performance on unseen tasks. In contrast to standard generalization bounds in multitask and meta-learning \cite{schnitzer2026probabilistic}, which typically depend only on sample size and model complexity, our bound captures the intrinsic difficulty of transferring across heterogeneous control tasks.
\end{remark}

Note that \eqref{generalization_error} can be naturally extended to the setting where a discrete set of training sets are sampled from $\mathcal{P}_{\mathcal{T}}$, namely, suppose that $\mathcal{S}_1, \ldots, \mathcal{S}_H$ are $H$ training sets, and we also have $H$ independent validation sets $\mathcal{S}_{\text{val},1}, \ldots, \mathcal{S}_{\text{val},H}$. Then, with probability $1-\delta^\prime - \delta$,
\begin{align*}
 \max_{\ell \in [H]}|&J(\widetilde{K}_{\infty,\mathcal{S}_\ell}) - \hat{J}(\widetilde{K}_{\infty,\mathcal{S}_{\ell}}) |\notag\\[-0.4cm]&\leq (J_{\star,\max} + \mu_{\max} b_{\max})\sqrt{\frac{\log(4H/(\delta^\prime+H\delta))}{2N_{\text{val}}}},
\end{align*}
with $J_{\star,\max}:= \max_{\ell \in  [H]} J_{\star, \mathcal{S}_{\ell}}$, $b_{\max}:= \max_{\ell \in  [H]} b_{\mathcal{S}_{\ell}}$, and $\mu_{\max}:= \max_{\ell \in  [H]} \mu_{\mathcal{S}_{\ell}}$. Note that above result follows from union bounding \eqref{generalization_error} over $H$ training sets. A rigorous extension of this result to an infinite collection of training sets (e.g., via covering arguments) is left for future work.

\section{Model-free Multitask LQG}

While the analysis so far has been conducted in a model-based setting, where system and cost matrices are known and policy gradients \eqref{eq:PG_iterates_multitask} can be computed exactly , learning a common controller from data can be achieved by replacing the exact gradients with estimators and controlling the resulting estimation error. Importantly, estimating policy gradients does not require explicit knowledge of the history representation, it suffices to access the history vectors $z^{(i)}_{t,p}$ for some history length $p \in \mathbb{N}$, together with an initial stabilizing controller (Assumption \ref{assump:stabilizing_controller}). This is discussed in detail in \cite{fallah2025gradient}.

\vspace{0.1cm}

\noindent \textbf{Gradient estimator.} The estimation of $\nabla \hat{J}(\widetilde{K})$ denoted by $\widehat{\nabla}J(\widetilde{K})$ is performed by using a one-point zeroth-order gradient estimation \cite{malik2019derivative}:
\begin{align}\label{gradient_estimation}
\widehat{\nabla}\hat{J}(\widetilde{K}) := \frac{1}{N}\sum_{i=1}^N \widehat \nabla J^{(i)}(\widetilde{K}),   
\end{align}
where $ \widehat{\nabla}J^{(i)}(\widetilde{K}) := \frac{d}{n_s}\sum_{m=1}^{n_s}\frac{\hat{J}^{(i)}_{\tau}(\widetilde{K} + U_{i,m})U_{i,m}}{r^2},$
with gradient dimensionality $d := n_up(n_u+n_y)$. Here $n_s$ denotes the number of samples, $r$ the smoothing radius, and $U_{i,m}$ the controller perturbation that is drawn uniformly from matrices with $\|U_{i,m}\|_F = r$. We define the cost estimation as
$$
\hat{J}^{(i)}_{\tau}(\widetilde{K}) := \frac{1}{n_c} \sum_{l=1}^{n_c} \widetilde{J}^{(i)}_\tau(\widetilde{K},\xi^l),
$$
where $n_c$ is the number of rollouts, $\xi^l$ denotes a realization of the process and measurement noises $\{w_t,v_t\}_{t=0}^{\tau-1}$, and $\tau$ is the horizon length. Define the truncated cost
$\widetilde{J}^{(i)}_\tau(\widetilde{K},\xi^l) := \frac{1}{\tau}\sum_{t=0}^{\tau-1} \big(y_t^{(i)\top} Q^{(i)} y_t^{(i)} + u_t^{(i)\top} R^{(i)} u_t^{(i)}\big),$
where the inputs are generated by the policy $\widetilde{K}$. The estimation error follows from \cite[Lemma 5]{zhao2024convergence}, extending to our stochastic setting amounts to conditioning the concentration inequalities on the noise.

Let $L:= \max_{\mathcal{T}^{(i)} \in \mathcal{S}} L_i$ and $\bar{J}:= \max_{\mathcal{T}^{(i)} \in \mathcal{S}} \bar{J}_i$.

\begin{lemma}[Variance Reduction]\label{lemma:variance_red} Given a fixed probability $\widetilde{\delta} \in (0,1)$, for any $\widetilde{K} \in \widetilde{\calK}_{\mathrm{stab},\mathcal{S}}$, it holds that
\begin{align*}
\left\|\widehat{\nabla}\hat{J}(\widetilde{K})\hspace{-0.1cm} -\hspace{-0.1cm} \nabla \hat{J}(\widetilde{K})\right\|_F^2
&\leq \frac{16 \bar{J}^2 \log(2/\widetilde{\delta})}{n_sN}\left(\frac{d}{r}\right)^2+2(Lr)^2.
\end{align*}

Moreover, by setting the smoothing radius according to $r =  \sqrt{\frac{d\bar{J}}{L}}\left(n_sN/\log(2/\widetilde{\delta})\right)^{-1/4}$, it holds that
\begin{align*}
\left\|\widehat{\nabla}\hat{J}(\widetilde{K}) \hspace{-0.1cm}-\hspace{-0.1cm} \nabla \hat{J}(\widetilde{K})\right\|_F^2
&\leq 18Ld\bar{J} \sqrt{\frac{\log(2/\widetilde{\delta})}{n_sN}}, \text{ w.p. } 1-\widetilde{\delta}.
\end{align*}
\end{lemma}
\begin{proof}
The proof follows from matrix Hoeffding's inequality \cite{vershynin2018high}, full details are deferred to Appendix \ref{app:proof_lemma_VII1}.
\end{proof}

\begin{corollary}[Model-free] \label{cor:model-free} Suppose that the conditions in Theorem \ref{theorem:algorithm_dependent} hold and that the policy iterates converge to $\widetilde{K}_{\infty,\mathcal{S}}$. In addition, suppose that policy gradients are estimated via \eqref{gradient_estimation} with smoothing radius set as in Lemma \ref{lemma:variance_red}. Then, with probability $1-\widetilde{\delta}$, where $\widetilde{\delta} \in (0,1)$, it holds:
\begin{align}\label{eq:gap_bound_K_infty_converges_model_free}
   J^{(i)}(\widetilde{K}_{\infty,\mathcal{S}})&-J^{(i)}(K_\star^{(i)})\leq \frac{3 \big\|\leo{S_p^{(i)}}\big\|^2\big\|\Sigma_{\widetilde{K}_\star^{(i)}}^{(i)}\big\| b_i(\widetilde{K}_{\infty,\mathcal{S}})}{2 \underline{\lambda_{\min}}_\Sigma \lambda_{\min}(R^{(i)})}\notag\\
   &+ \frac{108Ld\bar{J}}{\underline{\gamma_i}} \sqrt{\frac{\log(2/\widetilde{\delta})}{n_sN}}.
\end{align}
where $\underline{\lambda_{\min}}_\Sigma := \leo{\inf_{\widetilde{K} \in \widetilde{\mathcal{K}}_{\mathrm{stab},\mathcal{S}}}\lambda_{\min}(\Sigma_{\widetilde{K}}^{(i)})}^2$ and $\leo{\underline{\gamma_i}} := \inf_{\widetilde{K}\in \widetilde{\mathcal{K}}_{\mathrm{stab},\mathcal{S}}}\gamma_i(\widetilde{K})$ with $\gamma_i(\widetilde{K})$ being the gradient dominance coefficient from \eqref{gradient_dominance}.
\end{corollary}
The bound in \eqref{eq:gap_bound_K_infty_converges_model_free} highlights two key effects. The second term shows that multitask learning reduces the variance of the policy gradient estimator, as it scales with $1/\sqrt{n_s N}$, where aggregation across tasks effectively improves sample efficiency. This demonstrates the benefit of multitask learning in the model-free setting. On the other hand, the first term is a non-vanishing heterogeneity bias, governed by $b_i(\widetilde{K}_{\infty,\mathcal{S}})$, which persists even with infinitely many samples.

\section{Experiments}
\label{sec:numerical_results}

\begin{figure*}[t]
    \centering
    \begin{subfigure}[t]{0.31\textwidth}
        \centering
        \includegraphics[width=\textwidth]{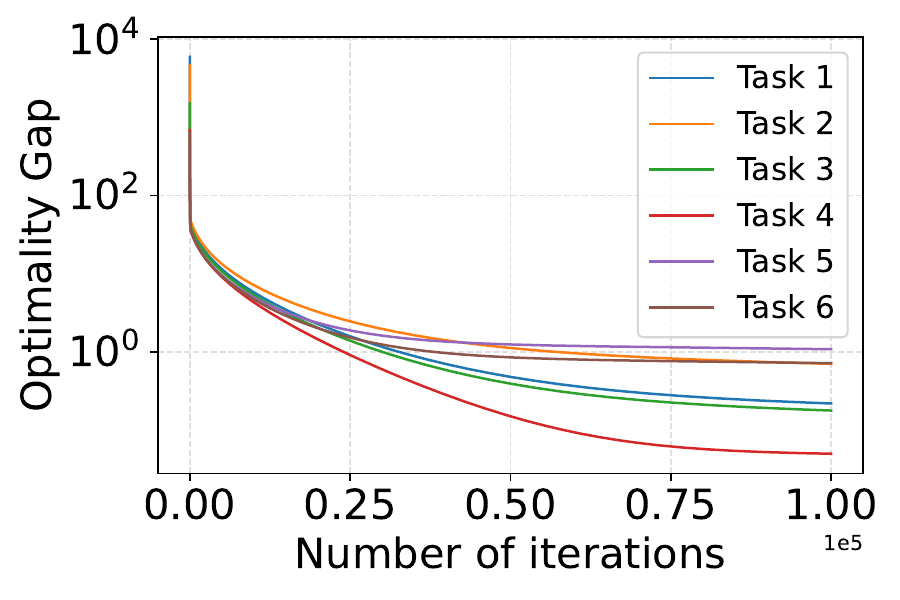}
        \label{fig:cartpole_log_gap}
    \end{subfigure}
    \hfill
    \begin{subfigure}[t]{0.30\textwidth}
        \centering
        \includegraphics[width=\textwidth]{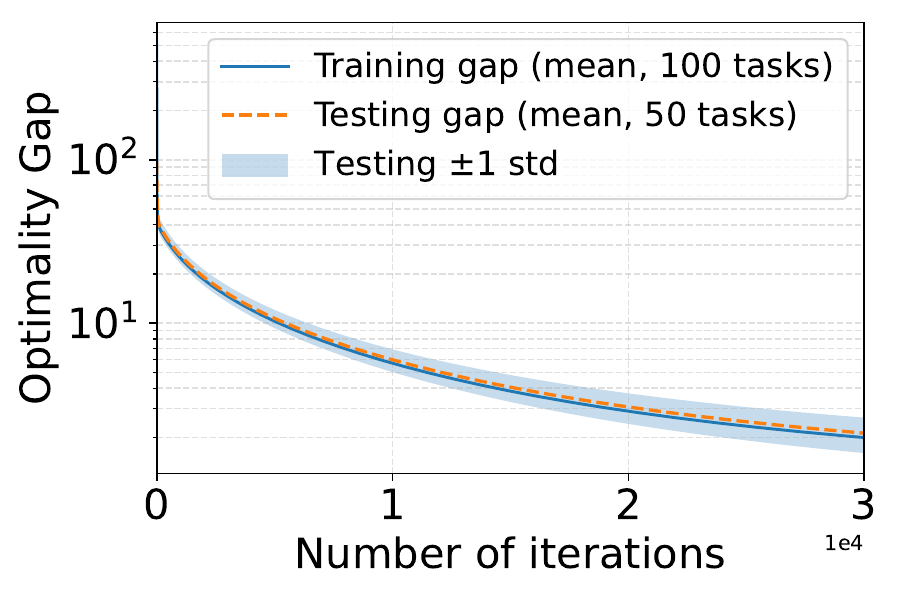}
        \label{fig:cartpole_generalization}
    \end{subfigure}
    \hfill
    \begin{subfigure}[t]{0.30\textwidth}
        \centering
        \includegraphics[width=\textwidth]{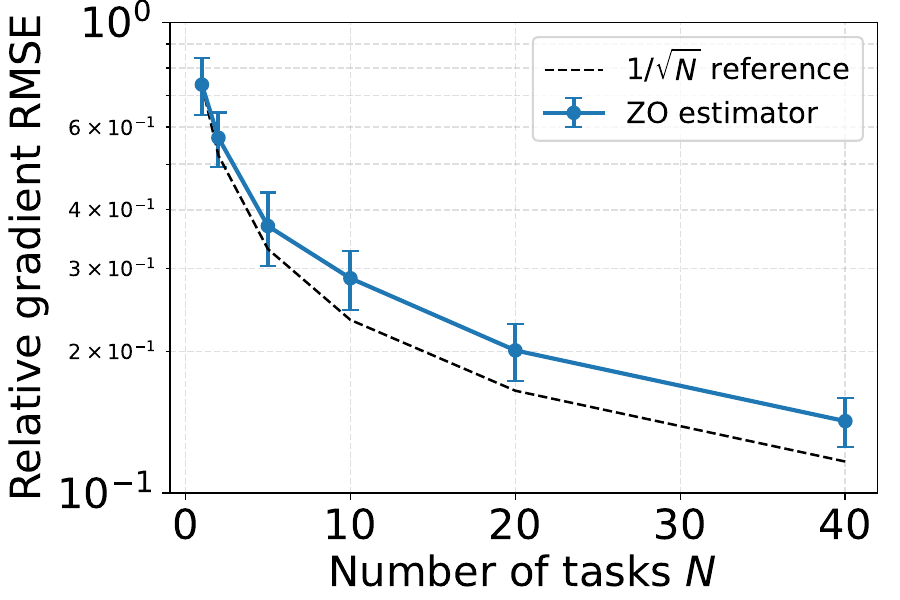}
        \label{fig:variance_reduction}
    \end{subfigure}
    \vspace{-0.7cm}
    \caption{\small
        Multitask LQG on partially observed cart-pole systems.
\text{(left)} Task-specific optimality gaps (first six training tasks) over iterations.
\text{(middle)} Train ($N=100$) and test ($50$) optimality gaps with $\pm 1$-std, showing strong generalization.
\text{(right)} Relative RMSE of the one-point ZO gradient estimator with respect to number of tasks $N$.
    }\vspace{-0.5cm}
    \label{fig:cartpole_four_panel}
\end{figure*}

We now evaluate multitask learning on a family of partially observed noisy cart-pole systems\footnote{Code: \url{https://github.com/jd-anderson/multitask_LQG}}. The training set consists of $N = 100$ tasks, while $50$ additional tasks drawn from the same distribution are used for evaluating the generalization.

\vspace{0.2cm}
\noindent \textbf{Task generation.} Each task $\mathcal{T}^{(i)}$ is generated by sampling the physical parameters of the cart-pole system from prescribed intervals. In particular, the pole mass $m_p^{(i)}$, cart mass $m_c^{(i)}$, and pole length $\ell^{(i)}$ are drawn from $[0.095, 0.105]$, $[0.95, 1.05]$, and $[0.475, 0.525]$, respectively. The cost matrices are constructed by sampling $q^{(i)}, r^{(i)}$ uniformly from $[0.095, 0.105]$ and setting
$ Q^{(i)} = q^{(i)} I_{n_x}$, $R^{(i)} = r^{(i)} I_{n_u}.$ Also, we set $W^{(i)} = 0.12 I_{n_x}$ and $V^{(i)} = 0.15 I_{n_y}$.

\vspace{0.2cm}
\noindent \textbf{System dynamics.} The discrete-time dynamics are obtained from the partially observed continuous-time cart-pole model \cite{wang2022learning} via forward Euler discretization with step-size $ 0.05$. 

\vspace{0.2cm}
\noindent \textbf{Training.} The shared lifted controller $\widetilde K$ is computed using \eqref{eq:PG_iterates_multitask}.
The initial controller is stabilizing, and the common multitask controller is trained for $100{,}000$ iterations with step-size $\alpha = 10^{-7}$.
\vspace{0.2cm}

\noindent \textbf{Results.} Figure \ref{fig:cartpole_four_panel} summarizes the experimental results on generalization, task-specific optimality gap, the bisimulation-based heterogeneity measure, and variance reduction.

Figure \ref{fig:cartpole_four_panel}-(left) reports the task-specific optimality gap for the first six tasks in the training set. All task-specific gaps decrease monotonically, while the residual gaps after a sufficiently large amount of iterations vary across tasks according to their heterogeneity as in Theorem \ref{theorem:algorithm_dependent}.

Figure \ref{fig:cartpole_four_panel}-(middle) depicts generalization across unseen tasks.
A common lifted controller is trained on $N=100$ tasks and evaluated on $50$ held-out tasks drawn from the same distribution $\mathcal{P}_{\mathcal{T}}$.
The train and test optimality gaps remain closely aligned throughout training, indicating that the learned controller generalizes well to unseen tasks.

Figure \ref{fig:cartpole_four_panel}-(right) illustrates the variance reduction achieved by the multitask gradient estimator as the number of training tasks increases.
The empirical RMSE decreases approximately at the $1/\sqrt{N}$ rate as implied in Lemma \ref{lemma:variance_red}, showing the benefit of multitask learning.

\section{Conclusions and Future Work} \label{sec:conclusions}

We studied the multitask LQG learning problem. By leveraging a history-dependent lifting, the multitask LQG problem is recast as high-dimensional multitask LQR, which allows for characterizing policy gradient bounds. We showed that task heterogeneity induces a bias that limits optimality, and we characterized this effect with approximate bisimulation. We provided policy gradient bounds and generalization guarantees that depend on the heterogeneity measure, and showed that multitask learning reduces gradient estimation error proportionally to the number of tasks used in training.

Future work includes leveraging the proposed heterogeneity measures to learn latent dynamics where multitask policy gradient updates are robust to task heterogeneity. Another direction is to improve gradient estimation in stochastic settings to further reduce estimation variance. Lastly, it would be of interest to characterize explicit heterogeneity conditions under which Assumption \ref{assump:uniform_stabilization} holds, providing when a common controller stabilizes all tasks in the support of $\mathcal{P}_{\mathcal{T}}$.

\section{Acknowledgments} 
Leonardo F. Toso is funded by the Center for AI and Responsible Financial Innovation (CAIRFI) Fellowship and the Columbia Presidential Fellowship. James Anderson is partially funded by NSF grants EECS 2144634 and CNS 2535097 and the Center of AI Technology (CAIT) in collaboration with Amazon. Charis Stamouli and George J. Pappas acknowledge support from NSF award EnCORE-2217058.

\bibliography{references_arxiv.bib}
\bibliographystyle{IEEEtran}
\newpage

\onecolumn

\newcommand{\R}{\mathbb{R}}
\newcommand{\PT}{\mathcal{P}_{\mathcal{T}}}
\newcommand{\Ktil}{\widetilde{K}}
\newcommand{\Jtil}{\widetilde{J}}
\newcommand{\normF}[1]{\left\|#1\right\|_F}
\newcommand{\spec}{\rho}           
\newcommand{\dare}{\operatorname{dare}}
\newcommand{\Sdag}{S^{(i)\dagger}_\star}
\newcommand{\SKi}{\Sigma^{(i)}_{\Ktil}}
\newcommand{\SKj}{\Sigma^{(j)}_{\Ktil}}
\newcommand{\AKi}{A^{(i)}_{\Ktil}}
\newcommand{\AKj}{A^{(j)}_{\Ktil}}
\newcommand{\EKi}{E^{(i)}_{\Ktil}}
\newcommand{\EKj}{E^{(j)}_{\Ktil}}
\newcommand{\PKi}{P^{(i)}_{\Ktil}}
\newcommand{\sKi}{s^{(i)}_{\Ktil,t}}
\newcommand{\sKj}{s^{(j)}_{\Ktil,t}}
\newcommand{\zKi}{z^{(i)}_{\Ktil,t}}
\newcommand{\zKj}{z^{(j)}_{\Ktil,t}}
\newcommand{\FKi}{F^{(i)}_{\Ktil}}
\newcommand{\FKj}{F^{(j)}_{\Ktil}}
\newcommand{\CKi}{C^{(i)}_{\Ktil}}
\newcommand{\CKj}{C^{(j)}_{\Ktil}}
\newcommand{\nui}{\nu^{(i)}}
\newcommand{\nuj}{\nu^{(j)}}
\newcommand{\ephet}{\epsilon_{\mathrm{het}}}
\newcommand{\bij}{b_{ij}(\Ktil)}
\newcommand{\Kstab}{\widetilde{\mathcal{K}}_{\mathrm{stab},\calS}}
\newcommand{\Kcal}{\widetilde{\mathcal{K}}_{\calS}}
\newcommand{\avg}{\textup{avg}}

\section{Appendix Roadmap}

This appendix provides detailed proofs, technical derivations, and additional experimental results supporting the main text. Section \ref{app:experiments} includes additional experimental details, such as system dynamics, task generation procedures, hyperparameters, and additional numerical results. In Section \ref{app:bisimulation}, we present the analysis of the bisimulation function, including proofs of existence as well as the computation of the proposed bisimulation-based heterogeneity measure via a semidefinite program. Section \ref{app:pg_bounds} contains the proofs of the policy gradient bounds, covering both algorithm-independent and algorithm-dependent guarantees, together with the role of task heterogeneity in controlling convergence. In Section \ref{app:generalization}, we establish the generalization bound using concentration inequalities, showing how the performance on unseen tasks relates to the empirical training cost. Finally, Section \ref{app:model_free} provides the analysis of the model-free setting, including variance bounds for the one-point zeroth-order gradient estimator and its implications for multitask variance reduction.

\section{Additional Experimental Details}
\label{app:experiments}

We begin by summarizing the system dynamics used in our numerical validation. In Section \ref{sec:numerical_results}, we consider a distribution of partially observed cart-pole tasks obtained by perturbing nominal physical parameters. These perturbations induce variability in both the system dynamics and cost parameters, thereby generating a distribution of LQG tasks.

\vspace{0.2cm}

\noindent \textbf{Cart-pole system.}
The cart-pole dynamics are generated from the linearized continuous-time model
$$
\dot{x}_t = A_c^{\mathrm{cp}} x_t + B_c^{\mathrm{cp}} u_t \text{ and }
y_t = C^{\mathrm{cp}} x_t,
$$
with state consisting of the cart position $x_t^{\mathrm{cart}}$, the cart velocity $\dot{x}_t^{\mathrm{cart}}$, the pole angle $\theta_t$ measured with respect to the upright position, and the angular velocity $\dot{\theta}_t$ of the pole. The system matrices are given by
$$
A_c^{\mathrm{cp}} =
\begin{bmatrix}
0 & 1 & 0 & 0 \\
0 & 0 & \dfrac{m_p}{m_c} g & 0 \\
0 & 0 & 0 & 1 \\
0 & 0 & \dfrac{m_p + m_c}{\ell m_c} g & 0
\end{bmatrix}, 
B_c^{\mathrm{cp}} =
\begin{bmatrix}
0 \\
\dfrac{1}{m_c} \\
0 \\
\dfrac{1}{\ell m_c}
\end{bmatrix}, \text{ and }
C^{\mathrm{cp}} =
\begin{bmatrix}
1 & 0 & 0 & 0 \\
0 & 1 & 0 & 1
\end{bmatrix}.
$$
The discrete-time matrices are obtained by forward Euler discretization with step size $dt$, namely
$$
A = I + dtA_c^{\mathrm{cp}}, \;
B = dtB_c^{\mathrm{cp}}, \text{ and }
C =
\begin{bmatrix}
1 & 0 & 0 & 0 \\
0 & 1 & 0 & 1
\end{bmatrix}.
$$

In the experiments, the nominal physical parameters are set to $m_p = 0.1$, $m_c = 1.0$, $\ell = 0.5$, $g = 9.81$, and $dt = 0.05$. In addition, different tasks are generated by perturbing $m_p$, $m_c$, and $\ell$ around their nominal values, thereby generating multiple partially observed cart-pole tasks with heterogeneous dynamics. In particular, the pole mass $m_p^{(i)}$, cart mass $m_c^{(i)}$, and pole length $\ell^{(i)}$ are drawn from $[0.095, 0.105]$, $[0.95, 1.05]$, and $[0.475, 0.525]$, respectively. The cost matrices are constructed by sampling $q^{(i)}, r^{(i)}$ uniformly from $[0.095, 0.105]$ and setting $ Q^{(i)} = q^{(i)} I_{n_x}$, $R^{(i)} = r^{(i)} I_{n_u}.$ Also, we set $W^{(i)} = 0.12 I_{n_x}$ and $V^{(i)} = 0.15 I_{n_y}$ for all tasks in the training set $\mathcal{S}$.

\subsection{Cart-pole - algorithm hyperparameters}
\vspace{0.2cm}
\begin{center}
\begin{tabular}{ll}
\hline
\textbf{Parameter} & \textbf{Value} \\ \hline
History length $p$ & 10 \\
Step-size $\alpha$ & $10^{-7}$ \\
Iterations & $10^5$ \\
Training tasks $N$ & 100 \\
Test tasks & 50 \\
Smoothing radius $r$ & $10^{-3}$ \\
Number of trajectories $n_s$ & 200 \\
Trajectory length $\tau$ & 200 \\
\hline
\end{tabular}
\end{center}

\vspace{0.2cm}
\noindent \textbf{Inverted-pendulum system.} We also consider a single-link inverted pendulum linearized on the upright equilibrium. The state is given by
$$
x_t =
\begin{bmatrix}
\theta_t \\
\dot{\theta}_t
\end{bmatrix},
$$
the input $u_t$ is the torque applied at the pivot, and the measured output is the angle, that is $y_t = \theta_t.$
The continuous-time dynamics are then given by
$$
\dot{x}_t = A_c^{\mathrm{ip}} x_t + B_c^{\mathrm{ip}} u_t,\;
y_t = C^{\mathrm{ip}} x_t,
$$
with system matrices
$$
A_c^{\mathrm{ip}} =
\begin{bmatrix}
0 & 1 \\
\dfrac{g}{\ell} & 0
\end{bmatrix}, 
B_c^{\mathrm{ip}} =
\begin{bmatrix}
0 \\
\dfrac{1}{m \ell^2}
\end{bmatrix}, \text{ and }
C^{\mathrm{ip}} =
\begin{bmatrix}
1 & 0
\end{bmatrix}.
$$

By using forward Euler discretization with sampling time $dt$, the discrete-time system becomes
$$
A = I + dt\,A_c^{\mathrm{ip}}=
\begin{bmatrix}
1 & dt \\
 \frac{dt g}{\ell} & 1
\end{bmatrix}, \;
B^{\mathrm{ip}} = dt B_c^{\mathrm{ip}}
=
\begin{bmatrix}
0 \\
\frac{dt}{m \ell^2}
\end{bmatrix},
\text{ and }
C^{\mathrm{ip}} =
\begin{bmatrix}
1 & 0
\end{bmatrix}.
$$

In the experiments, the nominal physical parameters are $m = 0.5, \; l = 0.3, \; g = 9.81, \text{ and } dt = 0.05.$ To generate different tasks, the pole mass $m_p^{(i)}$ and length $\ell^{(i)}$ are drawn from $[0.475, 0.525]$ and $[0.25, 0.35]$,  respectively. The cost matrices are constructed by sampling $q^{(i)}, r^{(i)}$ uniformly from $[0.095, 0.105]$ and setting $ Q^{(i)} = q^{(i)} I_{n_x}$, $R^{(i)} = r^{(i)} I_{n_u}.$ Also, we set $W^{(i)} = 0.02 I_{n_x}$ and $V^{(i)} = 0.05 I_{n_y}$ for all tasks in the training set $\mathcal{S}$.

\begin{figure}
    \centering
    \begin{subfigure}[t]{0.46\textwidth}
        \centering
        \includegraphics[width=\linewidth]{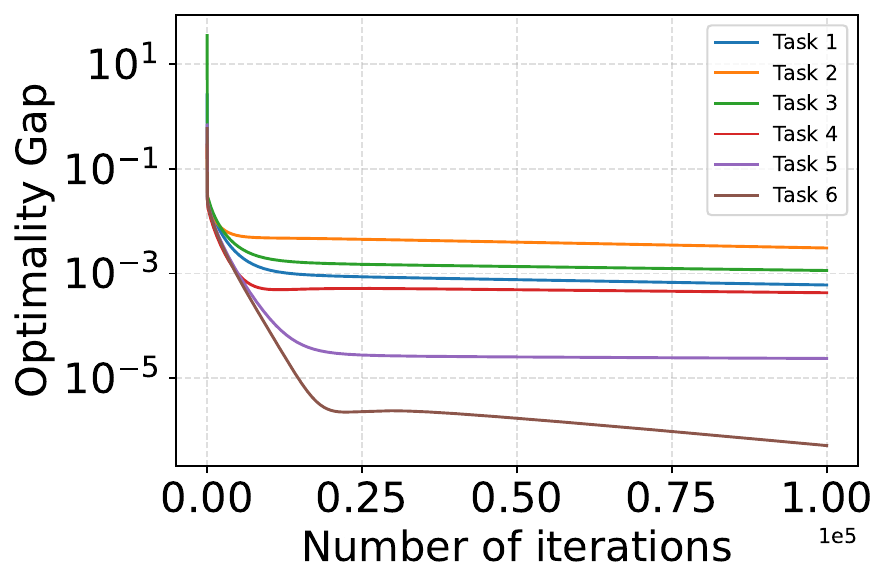}
        \caption{Task-specific optimality gaps.}
        \label{fig:pendulum_opt_gap}
    \end{subfigure}
    \hfill
    \begin{subfigure}[t]{0.46\textwidth}
        \centering
        \includegraphics[width=\linewidth]{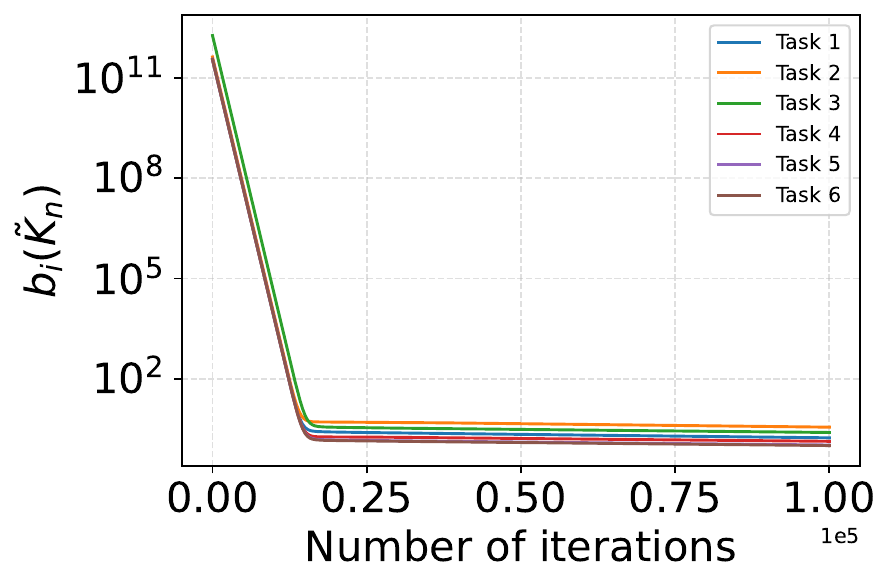}
        \caption{Bisimulation-based heterogeneity measure.}
        \label{fig:pendulum_bisim_measures}
    \end{subfigure}

    \vspace{0.5em}

    \begin{subfigure}[t]{0.46\textwidth}
        \centering
        \includegraphics[width=\linewidth]{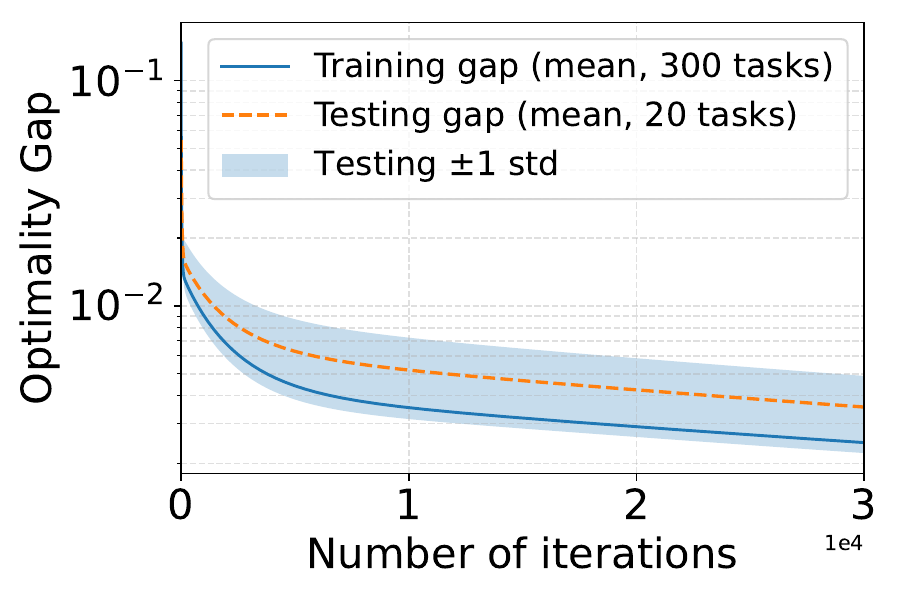}
        \caption{Training and testing optimality gaps.}
        \label{fig:pendulum_generalization}
    \end{subfigure}
    \hfill
    \begin{subfigure}[t]{0.46\textwidth}
        \centering
        \includegraphics[width=\linewidth]{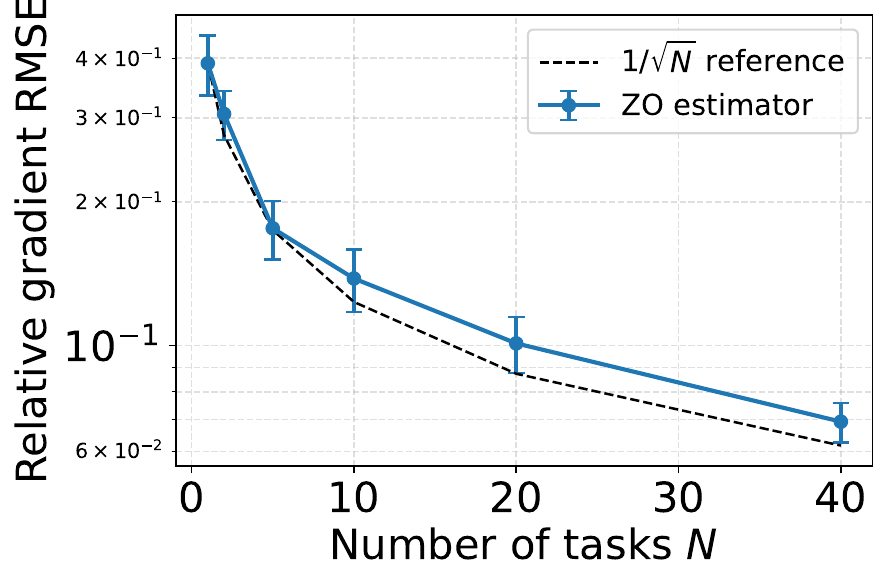}
        \caption{Relative gradient RMSE of the multitask one-point zeroth-order estimator as the number of tasks increases.}
        \label{fig:pendulum_variance_reduction}
    \end{subfigure}

    \caption{Additional numerical results for the partially observed inverted-pendulum task.  Top-left: task-specific optimality gaps.  Top-right: bisimulation-based heterogeneity measure.  Bottom-left: training and testing optimality gaps on held-out tasks, illustrating generalization.  Bottom-right: relative gradient RMSE versus the number of tasks, showing variance reduction.}
    \label{fig:pendulum_extra_results}
\end{figure}

\subsection{Inverted pendulum - algorithm hyperparameters}
\vspace{0.2cm}
\begin{center}
\begin{tabular}{ll}
\hline
\textbf{Parameter} & \textbf{Value} \\ \hline
History length $p$ & 12 \\
Step-size $\alpha$ & $10^{-2}$ \\
Iterations & $10^5$ \\
Training tasks $N$ & 300 \\
Test tasks & 20 \\
Smoothing radius $r$ & $0.01$ \\
Number of trajectories $n_s$ & 200 \\
Trajectory length $\tau$ & 200 \\
\hline
\end{tabular}
\end{center}

\vspace{0.2cm}

The numerical results are consistent with the theoretical predictions established in Section~\ref{sec:policy_gradient_bounds}. In particular, Figure \ref{fig:pendulum_extra_results} shows that the task-specific optimality gaps decrease over training iterations, while the remaining task-to-task discrepancies are captured by the bisimulation-based heterogeneity measures. Moreover, the shared lifted controller learned from the training tasks exhibits strong performance on held-out tasks sampled from the same distribution, indicating good generalization. Finally, in the model-free setting, the relative error of the one-point multitask gradient estimator decreases as the number of tasks increases, confirming the variance-reduction predicted in Lemma \ref{lemma:variance_red}.

\section{Auxiliary Results}

\noindent \textbf{Young's inequality:}  Given any two matrices $A, B \in \mathbb{R}^{n_x\times n_u}$,  for any $\beta>0$, we have

\begin{align}\label{eq:youngs}
\|A+B\|_2^2 &\leq(1+\beta)\|A\|_2^2+\left(1+\frac{1}{\beta}\right)\|B\|_2^2 \leq (1+\beta)\|A\|_F^2+\left(1+\frac{1}{\beta}\right)\|B\|_F^2.
\end{align}

Moreover, given any two matrices $A, B$ of the same dimensions,  for any $\beta>0$, we have
\begin{align}\label{eq:youngs_inner_product}
\langle A, B\rangle & \leq \frac{\beta}{2}\lVert A\rVert_2^2 +\frac{1}{2\beta}\lVert B \rVert_2^2  \leq  \frac{\beta}{2}\lVert A\rVert_F^2 +\frac{1}{2\beta}\lVert B \rVert_F^2.
\end{align}

\begin{theorem}[Cesàro Means - Chapter 5 of \cite{hardy2024divergent}] \label{theorem:cesaro}
Let $a_n \to a$, and let
\begin{align}
b_n = \frac{1}{n} \sum_{i=1}^{n} a_i,
\end{align}
then $
\lim_{n \to \infty} b_n = a.$
\end{theorem}

\section{Bisimulation Function Analysis}
\label{app:bisimulation}

\subsection{Proof of Lemma \ref{lemma:bisim_function}}
\label{app:proof_lemma_IV1}

\begin{proof}
Let $V:\R^{n_x^2}\times\R^{n_x^2}\to\R_+$ be a bisimulation function
between $\mathcal{S}^{(i)}_{\Ktil}$ and $\mathcal{S}^{(j)}_{\Ktil}$
as defined in Definition \ref{def:bisim_function}, with constants $\lambda^\prime \in (0,2),\zeta>0$. We begin, by applying condition \eqref{eq:BF_condition_2} recursively from $t=0$ to $t=T-1$:
\begin{align}
  V\left(s^{(i)}_{\Ktil,t},s^{(j)}_{\Ktil,t}\right)
  &\leq (1-\lambda')^t V \left(s^{(i)}_{\Ktil,0},s^{(j)}_{\Ktil,0}\right)
    + \zeta V(\nu^{(i)},\nu^{(j)})\sum_{k=0}^{t-1}(1-\lambda^\prime)^k \notag\\
  &\leq (1-\lambda^\prime)^t V_0
    + \frac{\zeta V(\nu^{(i)},\nu^{(j)})}{\lambda^\prime},
    \label{eq:V_bound}
\end{align}
where $V_0:=V(s^{(i)}_{\Ktil,0},s^{(j)}_{\Ktil,0})$ and we used
$\sum_{k=0}^\infty(1-\lambda')^k=1/\lambda'$. In addition, by condition \eqref{eq:BF_condition_1}, we have
$$
  \norm{z^{(i)}_{\Ktil,t}-z^{(j)}_{\Ktil,t}}^2
  \leq V \left(s^{(i)}_{\Ktil,t},s^{(j)}_{\Ktil,t}\right),
$$
and by taking the Cesàro mean (Theorem \ref{theorem:cesaro}) over $T$ steps and using \eqref{eq:V_bound}, we obtain
\begin{align*}
  \frac{1}{T}\sum_{t=0}^{T-1}
    \norm{z^{(i)}_{\Ktil,t}-z^{(j)}_{\Ktil,t}}^2
  &\leq \frac{1}{T}\sum_{t=0}^{T-1}V\left(s^{(i)}_{\Ktil,t},s^{(j)}_{\Ktil,t}\right)\\
  &\leq \frac{V_0}{T}\sum_{t=0}^{T-1}(1-\lambda')^t
    + \frac{\zeta\,V(\nu^{(i)},\nu^{(j)})}{\lambda'}\\
  &\leq \frac{V_0}{T\lambda'} + \frac{\zeta\,V(\nu^{(i)},\nu^{(j)})}{\lambda'}.
\end{align*}
Letting $T\to\infty$ and recalling
$\epsilon^{(ij)}_{\mathrm{het}}\left(\widetilde{K}\right)
 =\lim_{T\to\infty}\frac{1}{T}\sum_{t=0}^{T-1}
   \norm{z^{(i)}_{\Ktil,t}-z^{(j)}_{\Ktil,t}}^2$
(i.e., from \eqref{eq:convergence_covariance_system}) yields
$$
\epsilon^{(ij)}_{\mathrm{het}}\left(\widetilde{K}\right)
  \leq \frac{\zeta V(\nu^{(i)},\nu^{(j)})}{\lambda^\prime},
$$
which is exactly the bound in \eqref{eq:grad_het_V_bound}.
\end{proof}

\subsection{Proof of Lemma \ref{lemma:existence_M}}
\label{app:proof_lemma_IV2}
\begin{proof}
We need to show the existence of $M\in\mathbb{S}^{2n^2_x}$
and $\lambda\in(0,1)$ satisfying
\begin{align}
  &M \succeq C^{(ij)\top}_{\Ktil}C^{(ij)}_{\Ktil},
  \label{eq:M_lb}\\
  &F^{(ij)\top}_{\Ktil}MF^{(ij)}_{\Ktil} - M
    \preceq -\lambda M.
  \label{eq:M_decay}
\end{align}

We note that since $\Ktil\in\widetilde{\mathcal{K}}_{\mathrm{stab},\mathcal{S}}$ stabilizes all systems in the training set simultaneously, then
$A^{(i)}_{\Ktil}$ and $A^{(j)}_{\Ktil}$ are both Schur stable.  Hence,
$F^{(i)}_{\Ktil}=A^{(i)}_{\Ktil}\otimes A^{(i)}_{\Ktil}$ and
$F^{(j)}_{\Ktil}=A^{(j)}_{\Ktil}\otimes A^{(j)}_{\Ktil}$ are also Schur stable, by the fact that for any two square real-valued matrices $M_1$ and $M_2$, the spectral radius of the Kronecker product between the two matrices satisfies: $\rho(M_1\otimes M_2) \leq \rho(M_1)\rho(M_2)$. Then the block-diagonal matrix $F^{(ij)}_{\Ktil}$ is also Schur stable. 

\vspace{0.2cm}

\noindent  \textbf{Existence of $M$ satisfying \eqref{eq:M_decay}.}
Because $F^{(ij)}_{\Ktil}$ is Schur stable with spectral radius
$\rho<1$, there exists $\lambda\in(0,1)$ with $\rho<1-\lambda<1$
and a solution $M\succ 0$ to the inequality
$$
  F^{(ij)\top}_{\Ktil}MF^{(ij)}_{\Ktil} \preceq (1-\lambda)M.
$$
This follows from the standard discrete-time Lyapunov stability
theory, i.e., there exists $M\succ 0$ as the solution to the Lyapunov equation
$$
  F^{(ij)\top}_{\Ktil}MF^{(ij)}_{\Ktil} - M
  = -\lambda_0 I
$$
for any $\lambda_0>0$. Note that, because $M\succ 0$, we have that $M \preceq \|M\|I$. Therefore, by multiplying both sides by $\frac{\lambda_0}{\|M\|}$, we have that
$$
\frac{\lambda_0 M}{\|M\|} \preceq \frac{\lambda_0}{\|M\|}\|M\|I = \lambda_0 I,
$$
then, by choosing $\lambda = \frac{\lambda_0}{\|M\|}$, we obtain $\lambda M \preceq \lambda_0I$, which implies 
$$
  F^{(ij)\top}_{\Ktil}MF^{(ij)}_{\Ktil} - M
  \preceq -\lambda M,
$$
and thus we have the existence of $M \succ 0$ that satisfies \eqref{eq:M_decay}.

\vspace{0.2cm}

\noindent \noindent\textbf{Existence of $M$ satisfying \eqref{eq:M_lb}.}
Now given $M$ that satisfies \eqref{eq:M_decay}, if such $M$ does not already satisfy \eqref{eq:M_lb}, we can replace it
by $\gamma M$ for $\gamma>0$ large enough so that
$\gamma M\succeq C^{(ij)\top}_{\Ktil}C^{(ij)}_{\Ktil}$. This is true, because \eqref{eq:M_decay} is homogeneous in $M$, and thus the scaled matrix $\gamma M$ still satisfies \eqref{eq:M_decay} with the same fixed $\lambda$. In addition, if we choose $\gamma$ large enough such that $\gamma \lambda_{\min}(M) \geq \lambda_{\max}(C^{(ij)\top}_{\Ktil}C^{(ij)}_{\Ktil})$, then $\gamma M \succeq C^{(ij)\top}_{\Ktil}C^{(ij)}_{\Ktil}$.
\end{proof}

\subsection{Proof of Theorem \ref{theorem:Bisimulation_Functions_characterization}}
\label{app:proof_thm_IV1}

\begin{proof}
Let $M\succ 0$ and $\lambda\in(0,1)$ satisfy \eqref{eq:M_lb}-\eqref{eq:M_decay},
and fix $\eta\in(0,\lambda/(1-\lambda))$.
Define the quadratic function
$$
V\left(s^{(i)}_{\Ktil,t},s^{(j)}_{\Ktil,t}\right) := s^{(ij)\top}_{\Ktil,t}\,M\,s^{(ij)}_{\Ktil,t}, \text{ with } s^{(ij)}_{\Ktil,t}
:=\begin{bmatrix}s^{(i)}_{\Ktil,t}\\s^{(j)}_{\Ktil,t}\end{bmatrix}.
$$

By \eqref{eq:M_lb} and the Cauchy-Schwarz inequality, we obtain
\begin{align*}
  \norm{z^{(i)}_{\Ktil,t}-z^{(j)}_{\Ktil,t}}^2
  &= \norm{C^{(ij)}_{\Ktil}\,s^{(ij)}_{\Ktil,t}}^2
   \le s^{(ij)\top}_{\Ktil,t}
       C^{(ij)\top}_{\Ktil}C^{(ij)}_{\Ktil}
       s^{(ij)}_{\Ktil,t}\\
  &\leq s^{(ij)\top}_{\Ktil,t}Ms^{(ij)}_{\Ktil,t}
   := V\left(s^{(i)}_{\Ktil,t},s^{(j)}_{\Ktil,t}\right),
\end{align*}
satisfying condition \eqref{eq:BF_condition_1}.

On the other hand to verify \eqref{eq:BF_condition_2}, we recall that the one-step evolution of the joint state is given by
$$
s^{(ij)}_{\Ktil,t+1} = F^{(ij)}_{\Ktil}\,s^{(ij)}_{\Ktil,t} + \nu^{(ij)}, \text{ with } \nu^{(ij)}:=\begin{bmatrix}\nu^{(i)}\\\nu^{(j)}\end{bmatrix}.
$$
Hence
\begin{align} 
  V\left(s^{(i)}_{\Ktil,t+1},s^{(j)}_{\Ktil,t+1}\right)
   - V\left(s^{(i)}_{\Ktil,t},s^{(j)}_{\Ktil,t}\right) &=
    \left(F^{(ij)}_{\Ktil}s^{(ij)}_{\Ktil,t}+\nu^{(ij)}\right)^\top
    M
    \left(F^{(ij)}_{\Ktil}s^{(ij)}_{\Ktil,t}+\nu^{(ij)}\right)
    - s^{(ij)\top}_{\Ktil,t} M s^{(ij)}_{\Ktil,t} \notag\\
  &= s^{(ij)\top}_{\Ktil,t}
     \left(F^{(ij)\top}_{\Ktil}MF^{(ij)}_{\Ktil}-M\right)
    s^{(ij)}_{\Ktil,t}
    + 2 s^{(ij)\top}_{\Ktil,t} F^{(ij)\top}_{\Ktil} M \nu^{(ij)}
    + \nu^{(ij)\top}M\nu^{(ij)}. \label{eq:V_diff}
\end{align}
By using \eqref{eq:M_decay} on the first term, we obtain
$$
  s^{(ij)\top}_{\Ktil,t}
  \left(F^{(ij)\top}_{\Ktil}MF^{(ij)}_{\Ktil}-M\right)
  s^{(ij)}_{\Ktil,t}
  \leq -\lambda s^{(ij)\top}_{\Ktil,t}M s^{(ij)}_{\Ktil,t}
  = -\lambda V \left(s^{(i)}_{\Ktil,t},s^{(j)}_{\Ktil,t}\right).
$$

For the cross term, we apply Young's inequality. That is, for any $\eta>0$, it holds that $2a^\top b \leq \eta \|a\|^2 + \eta^{-1}\|b\|^2$, which gives the following expression.
\begin{align*}
   2s^{(ij)\top}_{\Ktil,t} F^{(ij)\top}_{\Ktil} M \nu^{(ij)}
  &\leq \frac{1}{\eta} \nu^{(ij)\top}M\nu^{(ij)}
    + \eta s^{(ij)\top}_{\Ktil,t}
      F^{(ij)\top}_{\Ktil}M F^{(ij)}_{\Ktil} s^{(ij)}_{\Ktil,t}\\
  &\leq  \frac{1}{\eta} \nu^{(ij)\top}M\nu^{(ij)}
    + \eta(1-\lambda) s^{(ij)\top}_{\Ktil,t}
      M s^{(ij)}_{\Ktil,t} \;\; (\text{using } \eqref{eq:constraint_stabilization}),
\end{align*}
with $a := M^{1/2}F^{(ij)}_{\Ktil} s^{(ij)}_{\Ktil,t}$ and $b := M^{1/2}\nu^{(ij)\top}$.  Therefore, combining with \eqref{eq:V_diff}
yields
$$
V\left(s^{(i)}_{\Ktil,t+1},s^{(j)}_{\Ktil,t+1}\right)
- V\left(s^{(i)}_{\Ktil,t},s^{(j)}_{\Ktil,t}\right)
  \leq -\lambda^\prime V\left(s^{(i)}_{\Ktil,t},s^{(j)}_{\Ktil,t}\right)
    + \zeta V(\nu^{(i)},\nu^{(j)}),
$$
with $\lambda^\prime:=\lambda-\eta(1-\lambda)>0$ (guaranteed by the
constraint $\eta<\lambda/(1-\lambda)$) and
$\zeta:=1+\eta^{-1}$.  This verifies condition \eqref{eq:BF_condition_2}.

In addition, by Lemma \ref{lemma:bisim_function}, we have
$$
\epsilon^{(ij)}_{\mathrm{het}}\left(\widetilde{K}\right) \leq \frac{\zeta V(\nu^{(i)},\nu^{(j)})}{\lambda^\prime} = \frac{\zeta \nu^{(ij)\top}M\nu^{(ij)}}{\lambda^\prime} = \frac{(1+\eta^{-1}) \nu^{(ij)\top}M\nu^{(ij)}}{\lambda - \eta(1-\lambda)},
$$
which is exactly \eqref{eq:grad_het_bis_bound}. Moreover, we note that the factor multiplying the quadratic function $\nu^{(ij)\top}M\nu^{(ij)}$, is the following function of $\eta$. 
$$
f(\eta):=\frac{1+\eta^{-1}}{\lambda - \eta(1-\lambda)} = \frac{\eta+1}{\eta\left(\lambda - \eta(1-\lambda)\right)}
$$

Note that $f(\eta)$ is well-defined for $\eta > 0$ such that $\lambda - \eta(1-\lambda) > 0$, i.e.,
$$
\eta \in \left(0, \frac{\lambda}{1-\lambda}\right).
$$

Let us define $N(\eta) := \eta + 1 \text{ and } D(\eta) := \eta\left(\lambda - \eta(1-\lambda)\right) = \lambda \eta - (1-\lambda)\eta^2.$ Then, we can write
\begin{align*}
f'(\eta) = \frac{N'(\eta)D(\eta) - N(\eta)D'(\eta)}{D(\eta)^2},
\end{align*}
where $N'(\eta) = 1$ and $D'(\eta) = \lambda - 2(1-\lambda)\eta$, which implies
\begin{align*}
f'(\eta)  = \frac{\lambda\eta - (1-\lambda)\eta^2 - (\eta+1)\left(\lambda - 2(1-\lambda)\eta\right)}{\left(\lambda\eta - (1-\lambda)\eta^2\right)^2}.
\end{align*}

Therefore, by expanding the numerator, we have
\begin{align*}
\lambda\eta - (1-\lambda)\eta^2  - \lambda\eta - \lambda 
+ 2(1-\lambda)\eta^2 + 2(1-\lambda)\eta = (1-\lambda)\eta^2 + 2(1-\lambda)\eta - \lambda,
\end{align*}
and then, we obtain
\begin{align*}
f'(\eta)  = \frac{(1-\lambda)\eta^2 + 2(1-\lambda)\eta - \lambda}{\left(\lambda\eta - (1-\lambda)\eta^2\right)^2}.
\end{align*}

Since the denominator is strictly positive for $\eta \in \left(0, \frac{\lambda}{1-\lambda}\right)$, the critical points are given by
\begin{align*}
(1-\lambda)\eta^2 + 2(1-\lambda)\eta - \lambda = 0.
\end{align*}

Then, by dividing by $1-\lambda > 0$, we obtain
\begin{align*}
\eta^2 + 2\eta - \frac{\lambda}{1-\lambda} = 0.
\end{align*}

By solving the quadratic equation, we have that
\begin{align*}
\eta = -1 \pm \sqrt{1 + \frac{\lambda}{1-\lambda}} = -1 \pm \sqrt{\frac{1}{1-\lambda}},
\end{align*}
and since $\eta > 0$, the only possible solution is given by
\begin{align}
\eta = \sqrt{\frac{1}{1-\lambda}} - 1.
\end{align}

Therefore, $f(\eta)$ is minimized by setting $\eta=\sqrt{1/(1-\lambda)}-1$, yielding the bisimulation-based heterogeneity measure $\bij$ in
\eqref{eq:bisim_het}.
\end{proof}

\subsection{Computation of Bisimulation-Based Task Heterogeneity}

We now provide the computation of the $M^{(ij)}_{\widetilde{K}}$ and
$\lambda^{(ij)}_{\widetilde{K}}$ that appear in the bisimulation-based heterogeneity measure in \eqref{eq:bisim_het}. These quantities result from solving the following optimization problem: 
\begin{subequations}\label{eq:opt_full}
\begin{align}
\min_{\lambda \in (0,1), M \succ 0, \eta \in (0,\frac{\lambda}{1-\lambda})} \quad
& \frac{\bigl(1 + \eta^{-1}\bigr)\,\nu^{(ij)\top} M\, \nu^{(ij)}}
       {\lambda - \eta(1-\lambda)}
\label{eq:opt_obj} \\
\text{s.t.} \quad
& M \succeq C^{(ij)\top}_{\widetilde{K}} C^{(ij)}_{\widetilde{K}},
\label{eq:opt_ca} \\
& F^{(ij)\top}_{\widetilde{K}} M F^{(ij)}_{\widetilde{K}} - M \preceq -\lambda M,
\label{eq:opt_cb}
\end{align}
\end{subequations}
where we choose $\eta$ optimally as discussed previously, i.e., 
$\eta = \eta^{(ij)}_{\widetilde{K}} := \sqrt{1/(1 \lambda^{(ij)}_{\widetilde{K}})} - 1$. Here the constraints represent the conditions of Lemma \ref{lemma:existence_M}. We begin by fixing $M$ and optimizing only with respect to $\lambda$.

\vspace{0.2cm}
\noindent \textbf{Optimizing over $\lambda$.} Note that the cost of~\eqref{eq:opt_full} decreases for larger values of~$\lambda$,
since $\eta = \sqrt{1/(1-\lambda)}-1$ is increasing in $\lambda$, so both
$(1+\eta^{-1})$ decreases and the denominator $\lambda - \eta(1-\lambda)$ increases, since $\lambda - \eta(1-\lambda) = 1 - \sqrt{1-\lambda}$. 
On the other hand, constraint \eqref{eq:opt_cb} is feasible if and only if the
eigenvalues of the rescaled matrix
$$
F^{(ij)}_{\widetilde{K},\lambda} := \frac{1}{\sqrt{1-\lambda}} F^{(ij)}_{\widetilde{K}}
$$
lie strictly within the unit disk, that is,
\begin{equation}\label{eq:lambda_feasibility}
   \rho \left(F^{(ij)}_{\widetilde{K},\lambda}\right) < 1
    \iff 
   \frac{\rho\left(F^{(ij)}_{\widetilde{K}}\right)}{\sqrt{1-\lambda}} < 1
    \iff 
   \lambda < 1 - \rho\left(F^{(ij)}_{\widetilde{K}}\right)^{2}.
\end{equation}
Recall that $F^{(ij)}_{\widetilde{K}} = \mathrm{diag}(F^{(i)}_{\widetilde{K}},  F^{(j)}_{\widetilde{K}})$,
where $F^{(i)}_{\widetilde{K}} := A^{(i)}_{\widetilde{K}} \otimes A^{(i)}_{\widetilde{K}}$. We can write
$$
\rho\left(F^{(ij)}_{\widetilde{K}}\right) = \rho\left(\begin{bmatrix}
    F^{(i)}_{\widetilde{K}} & 0 \\
    0 & F^{(j)}_{\widetilde{K}}
\end{bmatrix}\right) = \rho\left(\begin{bmatrix}
    F^{(i)}_{\widetilde{K}} & 0 \\
    0 & F^{(j)}_{\widetilde{K}}
\end{bmatrix}\right) =  \rho\left(\begin{bmatrix}
    A^{(i)}_{\widetilde{K}} \otimes A^{(i)}_{\widetilde{K}} & 0 \\
    0 & A^{(j)}_{\widetilde{K}} \otimes A^{(j)}_{\widetilde{K}}
\end{bmatrix}\right),
$$
which implies that 
$$
\rho\left(F^{(ij)}_{\widetilde{K}}\right) =  \rho\left(\begin{bmatrix}
    A^{(i)}_{\widetilde{K}} \otimes A^{(i)}_{\widetilde{K}} & 0 \\
    0 & A^{(j)}_{\widetilde{K}} \otimes A^{(j)}_{\widetilde{K}}
\end{bmatrix}\right) \leq \max\left(\rho \left(A^{(i)}_{\widetilde{K}}\right)^{2},\;
               \rho \left(A^{(j)}_{\widetilde{K}}\right)^{2}\right).
$$
by using the fact that $\rho(A \otimes A) = \rho(A)^{2}$. Since $\widetilde{K}$ is a common stabilizing controller,
$\rho(A^{(i)}_{\widetilde{K}}) < 1$ and $\rho(A^{(j)}_{\widetilde{K}}) < 1$, and
thus $1 - \rho(F^{(ij)}_{\widetilde{K}})^{2} \in (0,1)$.
Consequently, since the cost decreases monotonically in $\lambda$, the
optimal choice is given by
\begin{equation}\label{eq:lambda_opt}
   \lambda^{(ij)}_{\widetilde{K}} :=  1 - \rho\!\left(F^{(ij)}_{\widetilde{K}}\right)^{2} - \varepsilon,
\end{equation}
for some small constant $\varepsilon > 0$. We note that introducing $\varepsilon$ is standard when addressing optimization problems
with strict inequality constraints. From \eqref{eq:lambda_opt}, the corresponding
optimal $\eta$ is then
\begin{equation}\label{eq:eta_opt}
   \eta^{(ij)}_{\widetilde{K}}:= \frac{1}{\sqrt{\rho \left(F^{(ij)}_{\widetilde{K}}\right)^{2} + \varepsilon}} - 1.
\end{equation}

\noindent \textbf{Optimizing over $M$.} We now fix $\lambda := \lambda^{(ij)}_{\widetilde{K}}$ and the resulting constants
$$
   \zeta := 1 + \bigl(\eta^{(ij)}_{\widetilde{K}}\bigr)^{-1} \text{ and }
   \lambda' := \lambda^{(ij)}_{\widetilde{K}} - \eta^{(ij)}_{\widetilde{K}}\!\left(1-\lambda^{(ij)}_{\widetilde{K}}\right),
$$
and solve~\eqref{eq:opt_full} with respect to $M$. Since $\zeta$ and $\lambda'$ are now fixed positive scalars, minimizing~\eqref{eq:opt_obj} is equivalent to minimizing the linear objective given by
$$
   \nu^{(ij)\top} M\, \nu^{(ij)}
   \;=\;
   \mathrm{Tr}\!\left(M\,\nu^{(ij)}\nu^{(ij)\top}\right)
$$
over $M$. The resulting problem is
\begin{subequations}\label{eq:sdp}
\begin{align}
\min_{M \succeq \varepsilon I} \quad
& \mathrm{Tr}\!\left(M\,\nu^{(ij)}\nu^{(ij)\top}\right)
\label{eq:sdp_obj} \\
\text{s.t.} \quad
& M \succeq C^{(ij)\top}_{\widetilde{K}} C^{(ij)}_{\widetilde{K}},
\label{eq:sdp_ca} \\
& F^{(ij)\top}_{\widetilde{K}} M F^{(ij)}_{\widetilde{K}} - M \preceq -\lambda^{(ij)}_{\widetilde{K}} M.
\label{eq:sdp_cb}
\end{align}
\end{subequations}

We note that problem \eqref{eq:sdp} is a semidefinite program (SDP), i.e., the objective is linear in $M$ and all constraints are linear matrix inequalities (LMIs). In particular,
constraint~\eqref{eq:sdp_cb} can be rewritten as
$$
F^{(ij)\top}_{\widetilde{K}} M F^{(ij)}_{\widetilde{K}} \preceq  \left(1 - \lambda^{(ij)}_{\widetilde{K}}\right) M,
$$
which is an LMI in $M$ for fixed $\lambda^{(ij)}_{\widetilde{K}}$.
Therefore, problem \eqref{eq:sdp} is convex and can be solved efficiently using standard solvers, such as those in \textsc{CVXPY} \cite{diamond2016cvxpy}.

In summary, the bisimulation-based heterogeneity $b_{ij}(\widetilde{K})$ is computed
as follows:
\begin{enumerate}
\item Compute $F^{(i)}_{\widetilde{K}} = A^{(i)}_{\widetilde{K}} \otimes A^{(i)}_{\widetilde{K}}$, $C^{(i)}_{\widetilde{K}} = S^{(i)\dagger\top}_{\star} \otimes E^{(i)}_{\widetilde{K}}$, and $\nu^{(i)} = \mathrm{vec}(\Sigma^{(i)}_{\nu})$ for each task $i$, and form the joint quantities
$$
    F^{(ij)}_{\widetilde{K}} = \mathrm{diag} \left(F^{(i)}_{\widetilde{K}},\,F^{(j)}_{\widetilde{K}}\right),\;\
    C^{(ij)}_{\widetilde{K}} = \begin{bmatrix} C^{(i)}_{\widetilde{K}} & -C^{(j)}_{\widetilde{K}} \end{bmatrix}, \text{ and }
    \nu^{(ij)} = \begin{bmatrix} \nu^{(i)} \\ \nu^{(j)} \end{bmatrix}.
$$

\item Set $\lambda^{(ij)}_{\Ktil}$ and $\eta^{(ij)}_{\Ktil}$ via \eqref{eq:lambda_opt}-\eqref{eq:eta_opt}, and compute the derived constants $\zeta = 1 + (\eta^{(ij)}_{\widetilde{K}})^{-1}$ and $\lambda^\prime = \lambda^{(ij)}_{\widetilde{K}} - \eta^{(ij)}_{\widetilde{K}}(1 - \lambda^{(ij)}_{\widetilde{K}})$.

\item Solve the SDP \eqref{eq:sdp} to obtain $M^{(ij)}_{\widetilde{K}}$.

\item Evaluate the bisimulation-based heterogeneity measure via
$$
    b_{ij}(\widetilde{K}) := 
    \frac{\zeta \nu^{(ij)\top} M^{(ij)}_{\widetilde{K}} \nu^{(ij)}}{\lambda^\prime}.
$$
\end{enumerate}

\begin{remark}

It is important to emphasize the main difference between problem \eqref{eq:sdp} and the one in multitask LQR setting \cite{stamouli2025policy}. In that setting, the bisimulation measure involves the term $\sqrt{\lambda_{\min}(M)}$ in the denominator, which requires an epigraph reformulation and a second-order cone constraint to obtain a convex program. In the present LQG setting, the denominator of \eqref{eq:opt_obj} depends only on the scalar $\lambda^\prime$, and the numerator is linear in $M$. Consequently, the optimization over $M$
reduces to a SDP without the need for any additional reformulation.
This simplification arises because the bisimulation function is defined as the quadratic form $V(s^{(ij)}_{\widetilde{K},t}) = s^{(ij)\top}_{\widetilde{K},t} M s^{(ij)}_{\widetilde{K},t}$
on the vectorized gradient state, rather than the trace-based function used in the LQR matrix system as in \cite{stamouli2025policy}.    
\end{remark}

\section{Policy Gradient Bounds}
\label{app:pg_bounds}

\subsection{Proof of Theorem \ref{thm:algor_indep_bounds}}
\label{app:proof_thm_V1}

\begin{proof}

Let $\Ktil_{\star,\calS}$ denote the empirical minimizer of $\hat{J}(\Ktil)$. By the first-order optimality condition of $\hat{J}$, we have that
$$
  \nabla\hat{J}(\Ktil_{\star, \calS}) = \frac{1}{N}\sum_{i=1}^N
    \nabla J^{(i)}(\Ktil_{\star, \calS}) = 0.
$$

Moreover, by fixing any task in the training set $\calT^{(i)}\in\calS$, and   adding and subtracting $\nabla J^{(i)}(\Ktil_{\star, \calS})$, we have
\begin{align}
  \normF{\nabla J^{(i)}(\Ktil_{\star, \calS})}^2
  &= \normF{\nabla J^{(i)}(\Ktil_{\star, \calS}) - \frac{1}{N}\sum_{j=1}^N\nabla J^{(j)}(\Ktil_{\star,\calS})}^2\notag\\
  &\leq \left(\frac{1}{N}\sum_{j=1}^N
    \normF{\nabla J^{(i)}(\Ktil_{\star, \calS})
           -\nabla J^{(j)}(\Ktil_{\star, \calS})}\right)^2 \notag\\
  &\leq \frac{1}{N}\sum_{j=1}^N \normF{\nabla J^{(i)}(\Ktil_{\star, \calS})
           -\nabla J^{(j)}(\Ktil_{\star,\calS})}^2 = \frac{1}{N}\sum_{j=1}^N
      \epsilon^{(ij)}_{\mathrm{het}}\left(\Ktil_{\star, \calS}\right),
  \label{eq:grad_norm_bound}
\end{align}
where the second inequality is due to Jensen's inequality. Therefore, by applying Theorem \ref{theorem:Bisimulation_Functions_characterization}, we obtain
$$
\normF{\nabla J^{(i)}(\Ktil_{\star,\calS})}^2 \leq \frac{1}{N}\sum_{j\ne i}b_{ij}(\Ktil_{\star,\calS}),
$$
and by inserting into the gradient dominance condition \eqref{gradient_dominance}, we have
$$
  J^{(i)}(\Ktil_{\star,\calS}) - J^{(i)}(\Ktil^{(i)}_\star)
  \leq \frac{1}{\gamma_i(\Ktil_{\star,\calS})}\normF{\nabla J^{(i)}(\Ktil_{\star \calS})}^2
  \leq \frac{b_i(\Ktil_{\star,\calS})}{\gamma_i(\Ktil_{\star,\calS})}.
$$
Expanding $\gamma_i(\Ktil_{\star,\calS})$ gives the following expression
$$
  J^{(i)}(\Ktil^*_{\calS}) - J^{(i)}(\Ktil^{(i)}_\star)
  \leq \frac{\norm{\leo{S^{(i)}_p}}\norm{\Sigma^{(i)}_{\Ktil^{(i)}_\star}}b_i(\Ktil_{\star, \calS})}
       {4\lmin(\Sigma^{(i)}_{\Ktil_{\star,\calS}})^2\lmin(R^{(i)})},
$$
which completes the proof.
\end{proof}

\subsection{Proof of Theorem \ref{theorem:algorithm_dependent}}
\label{app:proof_thm_V2}

\begin{proof} The proof proceeds in three stages: (i) invariance of $\Kstab$ along the iterates, (ii) a per-iterate descent for the task-specific cost,
and (iii) bounding the task-specific optimality gap via the heterogeneity term.

\begin{lemma}[Generalized Bauer-Fike \cite{golub2013matrix}]\label{lemma:bauer-fike} Let $Q^\top A Q=D+S$ be the Schur decomposition of $A \in \mathbb{R}^{n_x\times n_x}$, where $D$ is diagonal and $S$ upper triangular with zeros in the diagonal. Then, it holds that
$$
|\rho(A+\Delta) - \rho(A)| \leq \max \left\{\|\Delta\|p_{\text{bf}}, \left(\|\Delta\|p_{\text{bf}}\right)^{1 / n_x}\right\}, \text{where $p_{\text{bf}} = \sum_{i=0}^{n_x-1}\|S\|^i$.}
$$
\end{lemma}

\vspace{0.2cm}

\noindent\textbf{Invariance of $\Kstab$.} To demonstrate that $\Ktil_n\in\Kstab$ for all iterations $n$ we use the following lemma.

\begin{lemma}[Conditions for Per-Iteration Stabilizing Controllers \cite{stamouli2025policy}]\label{lem:policy_gradient_stability}
Let $\widetilde{K}_0$ be a common stabilizing controller as in Assumption \ref{assump:stabilizing_controller} and consider the sequence $\{\widetilde{K}_n\}_{n=0}^\infty$ of policy gradient iterates given by \eqref{eq:PG_iterates_multitask}. Consider $\beta_1,\ldots,\beta_N\geq1$, $\omega_1,\ldots,\omega_N> 0$, $M_1, \ldots, M_N < \infty$, and $c_1,\ldots, c_N >0$, let $\Kstab = \cap_{i=1}^N\widetilde{\mathcal{K}}^{(i)}_{\mathrm{stab}}(\leo{\beta_i, \omega_i, M_i,c_i})$ denote the stabilizing subset of controllers in Definition \ref{def:stabilizing_set}. For each task $\mathcal{T}^{(i)} \in \mathcal{S}$, let $L_i$ denote the smoothness constant of the $i$-th cost function $J^{(i)}(\cdot)$ over the set $\Kstab$ and let $G_i$ denote the uniform bound of $\|\nabla J^{(i)}(\cdot)\|_F$ over the set $\Kstab$. We assume that the step size is such that 
\begin{align}
\alpha \leq \min\left\{\alpha_1, \alpha_2, \alpha_3, \alpha_4, \alpha_5 \right\},
\end{align}
$\alpha_1:= \frac{\max_{\mathcal{T}^{(i)} \in \mathcal{S}} M_i}{2\left(\frac{1}{N}\sum_{i=1}^N G_i\right)}$, $\alpha_2:= \frac{(\min_{\mathcal{T}^{(i)} \in \mathcal{S}}\omega_i/2)^{n_x}}{ \left(\max_{\mathcal{T}^{(i)} \in \mathcal{S}}\|B^{(i)}\|_F \leo{\|S^{(i)\dagger}_p\|_F} \right) \left(\frac{1}{N}\sum_{i=1}^N G_i\right)\left(\sum_{l=0}^{n_x-1} \bar{S}^l \right)}$, $\alpha_3 := \frac{1}{\max_{\mathcal{T}^{(i)} \in \mathcal{S}}\|B^{(i)}\|_F\|S^{(i)\dagger}_p\|_F\left(\frac{1}{N}\sum_{i=1}^N G_i\right)}$, $\alpha_4 := \frac{\min_{\mathcal{T}^{(i)} \in \mathcal{S}} \inf_{\widetilde{K} \in \Kstab} \sigma_{\min}(\mathcal{C}(A^{(i)}_{\widetilde{K}}))}{2\sqrt{\max_{\mathcal{T}^{(i)}\in \mathcal{S}}\|\Sigma^{(i)1/2}_\nu\|^2_F\left(\max_{\mathcal{T}^{(i)} \in \mathcal{S}} \left(\| B^{(i)} \|_F \|\leo{S^{(i)\dagger}_p}\|_F\right)\right)^2 \left(\frac{1}{N}\sum_{i = 1}^N G^2_i\right) \sum_{l = 1}^{n_x - 1} l^2 \left(\bar{S} +1\right)^{2l-2}}}$, $\alpha_5 := \frac{1}{4\max_{\mathcal{T}^{(i)} \in \mathcal{S}}L_i}$, and $\alpha_6 := \frac{8}{\max_{\mathcal{T}^{(i)} \in \mathcal{S}} \sup_{\widetilde{K} \in \Kstab} \gamma_{i}(\widetilde{K})} $. Here $\bar{S} := \max_{\mathcal{T}^{(i)} \in \mathcal{S}} \|A^{(i)}\|_F + M_i\|B^{(i)}\|_F\|\leo{S^{(i)\dagger}_p}\|_F$, $\mathcal{C}({A}^{(i)}_{\widetilde{K}}) := \left[ \Sigma^{(i)1/2}_\nu \; {A}^{(i)}_{\widetilde{K}}\Sigma^{(i)1/2}_\nu \; \ldots \; {A}^{(i)n_x - 1}_{\widetilde{K}}\Sigma^{(i)1/2}_\nu \right]$, and  ${A}^{(i)}_{\widetilde{K}} :=  A^{(i)} - B^{(i)}\widetilde{K} S^{(i)\dagger}_p$.

Moreover, we assume we operate on the low heterogeneity regime:
\begin{equation}\label{eq:heter_assumption}
\leo{\sup_{\widetilde{K}\in\Kstab}b_i(\widetilde{K})\leq \frac{\inf_{\widetilde{K}\in\Kstab} \gamma_i(\widetilde{K})}{6}(J^{(i)}( \widetilde{K}_0)-J^{(i)}( \widetilde{K}_\star^{(i)})),}
\end{equation}
where $\gamma_i(\widetilde{K})=4 \lambda_{\min}(\Sigma_{\widetilde{K}}^{(i)})^2 \lambda_{\min}(R^{(i)})/(\|\Sigma_{\widetilde{K}^{(i)}_\star}^{(i)}\| \|\leo{S^{(i)}_p}\|^2)$. Then, the line segment connecting $\widetilde{K}_n$ and $\widetilde{K}_{n+1}$ is contained in $\Kstab$, for all iterations $n\in\setN$.
\end{lemma}

\begin{proof}
We begin by recalling that from Assumption \ref{assump:stabilizing_controller}, we have that $\widetilde{K}_0 \in \Kstab := \cap_{i=1}^N\widetilde{\mathcal{K}}^{(i)}_{\mathrm{stab}}(\leo{\beta_i, \omega_i, M_i,c_i})$, for some $\beta_i \geq 1$, $\omega_i >0$, and $M_i < \infty$, for every task $\mathcal{T}^{(i)} \in \mathcal{S}.$ Let us define the following subset of stabilizing controllers

\begin{align}
\widetilde{\mathcal{D}}_{\mathrm{stab},\mathcal{S}}:= \cap_{i=1}^N \widetilde{\mathcal{D}}^{(i)}_{\mathrm{stab}}(\omega_i,M_i):= \cap_{i=1}^N \left\{\widetilde{K}^{(i)} \mid  \leo{\rho\left(A^{(i)} - B^{(i)} \widetilde{K}^{(i)}\leo{S^{(i)\dagger}_p}\right) \leq 1-\omega_i} \text{ and } \leo{\|\widetilde{K}^{(i)}\|_F \leq M_i} \right\},
\end{align}
where $\Kstab \subseteq \widetilde{\mathcal{D}}_{\mathrm{stab},\mathcal{S}}$. We then proceed to demonstrate that the line segment $\widetilde{K}_0(s):= \widetilde{K}_0 - s\alpha \nabla \hat{J}(\widetilde{K}_0)$, for all $s \in [0,1]$, also belongs to $\widetilde{\mathcal{D}}_{\mathrm{stab},\mathcal{S}}$.\\

\noindent \textbf{Spectral radius condition.} We need to demonstrate that $\rho\left(A^{(i)} - B^{(i)} \widetilde{K}_0(s)\leo{S^{(i)\dagger}_p}\right) < 1$ for all tasks $\mathcal{T}^{(i)} \in \mathcal{S}$. For this purpose, we can first write 
\begin{align}
\rho\left(A^{(i)} - B^{(i)} \widetilde{K}_0(s)\leo{S^{(i)\dagger}_p}\right) &= \rho\left(A^{(i)} - B^{(i)}  \widetilde{K}_0\leo{S^{(i)\dagger}_p} + \alpha s B^{(i)}\nabla \hat{J}(\widetilde{K}_0)\leo{S^{(i)\dagger}_p} \right)\notag \\
&\leq \rho\left(A^{(i)} - B^{(i)}  \widetilde{K}_0\leo{S^{(i)\dagger}_p}\right) + \max \left\{\|\alpha s B^{(i)} \nabla \hat{J}(\widetilde{K}_0)\leo{S^{(i)\dagger}_p}\|p^{(i)}_{\text{bf}}, \left(\|\alpha s B^{(i)} \nabla \hat{J}(\widetilde{K}_0)\leo{S^{(i)\dagger}_p}\|p^{(i)}_{\text{bf}}\right)^{1 / n_x}\right\}\notag\\
&\leq 1-\omega_i + \max \left\{\|\alpha s B^{(i)} \nabla \hat{J}(\widetilde{K}_0)\leo{S^{(i)\dagger}_p}\|p^{(i)}_{\text{bf}}, \left(\|\alpha s B^{(i)} \nabla \hat{J}(\widetilde{K}_0)\leo{S^{(i)\dagger}_p}\|p^{(i)}_{\text{bf}}\right)^{1 / n_x}\right\}
\end{align}
where the first inequality follows from Lemma \ref{lemma:bauer-fike},  with $p^{(i)}_{\text{bf}} = \sum_{l=0}^{n_x-1}\|S^{(i)}\|^l$ and $\|S^{(i)}\| \leq \|S^{(i)}\|_F \leq  \|A^{(i)} - B^{(i)} \widetilde{K}_0\leo{S^{(i)\dagger}_p}\|_F$. Hence, we have $\|S^{(i)}\| \leq  \|A^{(i)}\|_F + \|B^{(i)}\|_F \|\widetilde{K}_0\|_F\|\leo{S^{(i)\dagger}_p}\|_F \leq \max_{\mathcal{T}^{(i)} \in \mathcal{S}} \|A^{(i)}\|_F + M_i\|B^{(i)}\|_F\|\leo{S^{(i)\dagger}_p}\|_F := \bar{S}$, which implies that $p^{(i)}_{\text{bf}} \leq \sum_{l=0}^{n_x-1} \bar{S}^l$. Therefore, by selecting the step size such that 
\begin{align}
\alpha \leq \frac{(\min_{\mathcal{T}^{(i)} \in \mathcal{S}}\omega_i/2)^{n_x}}{ \left(\max_{\mathcal{T}^{(i)} \in \mathcal{S}}\|B^{(i)}\|_F \leo{\|S^{(i)\dagger}_p\|_F} \right) \left(\frac{1}{N}\sum_{i=1}^N G_i\right)\left(\sum_{l=0}^{n_x-1} \bar{S}^l \right)},
\end{align}
we obtain 
\begin{align}
\rho\left(A^{(i)} + B^{(i)} \widetilde{K}_0(s)\leo{S^{(i)\dagger}_p}\right) \leq 1-\frac{\omega_i}{2} = 1-\omega^\prime_i,
\end{align}
and thus, $\rho\left(A^{(i)} + B^{(i)} \widetilde{K}_0(s)\leo{S^{(i)\dagger}_p}\right) < 1$, for all tasks $\mathcal{T}^{(i)} \in \mathcal{S}$ and all $s \in [0,1]$.\\

\noindent \textbf{Bound on the norm of the lifted controller.} We aim to demonstrate that $\|\widetilde{K}_0(s)\|_F < \infty$ for all $s \in [0,1]$. To do so, we write
\begin{align*}
    \|\widetilde{K}_0(s)\|_F = \|\widetilde{K}_0 - \alpha s \nabla \hat{J}(\widetilde{K}_0)\|_F &\leq \|\widetilde{K}_0\|_F + \alpha s \|\nabla \hat{J}(\widetilde{K}_0)\|_F\\
    &\leq \max_{\mathcal{T}^{(i)} \in \mathcal{S}} M_i + \alpha s \|\nabla \hat{J}(\widetilde{K}_0)\|_F\\
    &\leq \max_{\mathcal{T}^{(i)} \in \mathcal{S}} M_i + \alpha s \left(\frac{1}{N}\sum_{i=1}^N G_i\right) \\
    &\leq \frac{3}{2}\max_{\mathcal{T}^{(i)} \in \mathcal{S}} M_i := M^\prime < \infty,
\end{align*}
where the last inequality follows from $\alpha \leq \frac{\max_{\mathcal{T}^{(i)} \in \mathcal{S}} M_i}{2\left(\frac{1}{N}\sum_{i=1}^N G_i\right)}$. Therefore, we demonstrated that for a sufficiently small step size, $\alpha$, and a given $\widetilde{K}_0$ in a compact set $\widetilde{\mathcal{D}}_{\mathrm{stab},\mathcal{S}}$ such that $\rho\left(A^{(i)} - B^{(i)} \widetilde{K}_0\leo{S^{(i)\dagger}_p}\right) < 1 \text{ and } \leo{\|\widetilde{K}_0\|_F < \infty}$, then $\widetilde{K}_0(s)$, for every $s \in [0,1]$, does also belong to a compact set where $\rho\left(A^{(i)} - B^{(i)} \widetilde{K}_0(s)\leo{S^{(i)\dagger}_p}\right) < 1 \text{ and } \leo{\|\widetilde{K}_0(s)\|_F < \infty}$ for all $\mathcal{T}^{(i)} \in \mathcal{S}$.\\

\noindent \textbf{Controllability.} From Assumption \ref{assump:stabilizing_controller}, we have that $(A^{(i)} - B^{(i)}\widetilde{K}_0 S^{(i)\dagger}_p, \Sigma^{(i)1/2}_\nu)$ is controllable for all tasks $\mathcal{T}^{(i)} \in \mathcal{S}$. Hence, we proceed to demonstrate that $(A^{(i)} - B^{(i)}\widetilde{K}_0(s) S^{(i)\dagger}_p, \Sigma^{(i)1/2}_\nu)$ is also controllable for every $s \in [0,1]$ and every $\mathcal{T}^{(i)} \in \mathcal{S}$. Let us first define ${A}^{(i)}_{\widetilde{K}_0} :=  A^{(i)} - B^{(i)}\widetilde{K}_0 S^{(i)\dagger}_p$ and $\mathcal{C}({A}^{(i)}_{\widetilde{K}_0}) := \left[ \Sigma^{(i)1/2}_\nu \; {A}^{(i)}_{\widetilde{K}_0}\Sigma^{(i)1/2}_\nu \; \ldots \; {A}^{(i)n_x - 1}_{\widetilde{K}_0}\Sigma^{(i)1/2}_\nu \right]$. 

We can write 

\begin{align}\label{eq:diference_in_A}
\left\|{A}^{(i)}_{\widetilde{K}_0} - {A}^{(i)}_{\widetilde{K}_0(s)}\right\|_F \leq \|\alpha s B^{(i)} \nabla \hat{J}(\widetilde{K}_0) \leo{S^{(i)\dagger}_p}\| &\leq \alpha \| B^{(i)} \|_F \|\nabla \hat{J}(\widetilde{K}_0)\|_F \|\leo{S^{(i)\dagger}_p}\|_F \notag \\
&\leq \alpha \max_{\mathcal{T}^{(i)} \in \mathcal{S}} \left(\| B^{(i)} \|_F \|\leo{S^{(i)\dagger}_p}\|_F\right) \left(\frac{1}{N}\sum_{i = 1}^N G_i\right).
\end{align}

In addition, we have 
\begin{align*}
&\left\|\mathcal{C}\left({A}^{(i)}_{\widetilde{K}_0}\right)-\mathcal{C}\left({A}^{(i)}_{\widetilde{K}_0(s)}\right)\right\|^2_F = \sum_{l=1}^{n_x - 1}\left\|\left( {A}^{(i) l}_{\widetilde{K}_0} - {A}^{(i)l}_{\widetilde{K}_0(s)}\right) \Sigma^{(i)1/2}_\nu\right\|^2_F\\
&\leq \|\Sigma^{(i)1/2}_\nu\|^2_F\|{A}^{(i)}_{\widetilde{K}_0(s)} - {A}^{(i)}_{\widetilde{K}_0}\|^2_F \sum_{l = 1}^{n_x - 1} l^2 \left(\max\left\{\|{A}^{(i)}_{\widetilde{K}_0(s)}\|_F, \|{A}^{(i)}_{\widetilde{K}_0}\|_F\right\}\right)^{2l - 2}\\
&\leq \|\Sigma^{(i)1/2}_\nu\|^2_F\|{A}^{(i)}_{\widetilde{K}_0(s)} - {A}^{(i)}_{\widetilde{K}_0}\|^2_F \sum_{l = 1}^{n_x - 1} l^2 \left(\|{A}^{(i)}_{\widetilde{K}_0}\|_F +1\right)^{2l-2}\\
&\leq \alpha^2 \max_{\mathcal{T}^{(i)}\in \mathcal{S}}\|\Sigma^{(i)1/2}_\nu\|^2_F\left(\max_{\mathcal{T}^{(i)} \in \mathcal{S}} \left(\| B^{(i)} \|_F \|\leo{S^{(i)\dagger}_p}\|_F\right)\right)^2 \left(\frac{1}{N}\sum_{i = 1}^N G^2_i\right) \sum_{l = 1}^{n_x - 1} l^2 \left(\max_{\mathcal{T}^{(i)}\in \mathcal{S}}\|{A}^{(i)}_{\widetilde{K}_0}\|_F +1\right)^{2l-2}\\
&\leq \alpha^2 \max_{\mathcal{T}^{(i)}\in \mathcal{S}}\|\Sigma^{(i)1/2}_\nu\|^2_F\left(\max_{\mathcal{T}^{(i)} \in \mathcal{S}} \left(\| B^{(i)} \|_F \|\leo{S^{(i)\dagger}_p}\|_F\right)\right)^2 \left(\frac{1}{N}\sum_{i = 1}^N G^2_i\right) \sum_{l = 1}^{n_x - 1} l^2 \left(\bar{S} +1\right)^{2l-2},
\end{align*}
where the second inequality follows from $\alpha \leq \frac{1}{\max_{\mathcal{T}^{(i)} \in \mathcal{S}}\|B^{(i)}\|_F\|S^{(i)\dagger}_p\|_F\left(\frac{1}{N}\sum_{i=1}^N G_i\right)}$. Moreover, the third inequality is due to \eqref{eq:diference_in_A}. Therefore, by using Weyl's singular value perturbation inequality, we can write
\begin{align*}
\sigma_{\min}(\mathcal{C}(A_{\widetilde{K}_0(s)})) &\geq \underbrace{\sigma_{\min}(\mathcal{C}(A_{\widetilde{K}_0}))}_{:=\xi_0 > 0} - \left\|\mathcal{C}\left({A}^{(i)}_{\widetilde{K}_0}\right)-\mathcal{C}\left({A}^{(i)}_{\widetilde{K}_0(s)}\right)\right\|\\
&\geq \xi_0 - \alpha \sqrt{\max_{\mathcal{T}^{(i)}\in \mathcal{S}}\|\Sigma^{(i)1/2}_\nu\|^2_F\left(\max_{\mathcal{T}^{(i)} \in \mathcal{S}} \left(\| B^{(i)} \|_F \|\leo{S^{(i)\dagger}_p}\|_F\right)\right)^2 \left(\frac{1}{N}\sum_{i = 1}^N G^2_i\right) \sum_{l = 1}^{n_x - 1} l^2 \left(\bar{S} +1\right)^{2l-2}}\\
&\geq \frac{\xi_0}{2} >0,
\end{align*}
where the last inequality is due to selecting the step size such that 
$$
\alpha \leq \frac{\xi_0}{2\sqrt{\max_{\mathcal{T}^{(i)}\in \mathcal{S}}\|\Sigma^{(i)1/2}_\nu\|^2_F\left(\max_{\mathcal{T}^{(i)} \in \mathcal{S}} \left(\| B^{(i)} \|_F \|\leo{S^{(i)\dagger}_p}\|_F\right)\right)^2 \left(\frac{1}{N}\sum_{i = 1}^N G^2_i\right) \sum_{l = 1}^{n_x - 1} l^2 \left(\bar{S} +1\right)^{2l-2}}}.
$$

This implies that under a sufficiently small step size the line segment $\widetilde{K}_0(s)$ also satisfy the controllability condition from Assumption \ref{assump:stabilizing_controller} for every $s \in [0,1]$ and every task $\mathcal{T}^{(i)} \in \mathcal{S}.$ 

\vspace{0.2cm}

\noindent \textbf{Optimality gap.} Moreover, given that $\rho\left(A^{(i)} - B^{(i)} \widetilde{K}_0(s)\leo{S^{(i)\dagger}_p}\right) < 1 \text{ and } \leo{\|\widetilde{K}_0(s)\|_F < \infty}$ for all $\mathcal{T}^{(i)} \in \mathcal{S}$ and $s \in [0,1]$, we can write  
\begin{align*}
J^{(i)}(\widetilde{K}_0(s)) - J^{(i)}(\widetilde{K}_0) &\leq \langle \nabla J^{(i)}(\widetilde{K}_0),\widetilde{K}_0(s) -\widetilde{K}_0\rangle  + \frac{L_i}{2}\|\widetilde{K}_0(s) -\widetilde{K}_0\|^2_F\\
&\leq \langle \nabla J^{(i)}(\widetilde{K}_0),-\alpha s \nabla \hat{J}(\widetilde{K}_0)\rangle  + \frac{L_i \alpha^2 s^2}{2}\|\nabla \hat{J}(\widetilde{K}_0)\|^2_F\\
&\leq \left(L_i\alpha^2s -\frac{\alpha s}{2}\right)\|\nabla J^{(i)}(\widetilde{K}_0)\|^2_F + \left(\frac{\alpha s}{2} + L_i \alpha^2s\right)\|\nabla J^{(i)}(\widetilde{K}_0) - \nabla \hat{J}(\widetilde{K}_0)\|^2_F\\
&\leq -\frac{\alpha s}{4}\|\nabla J^{(i)}(\widetilde{K}_0)\|^2_F + \frac{3\alpha s}{4}\|\nabla J^{(i)}(\widetilde{K}_0) - \nabla \hat{J}(\widetilde{K}_0)\|^2_F \\
&\leq -\frac{\alpha s \gamma_i(\widetilde{K}_0)}{4}\|\nabla J^{(i)}(\widetilde{K}_0)\|^2_F + \frac{3\alpha s}{4}b_i(\widetilde{K}_0),
\end{align*}
where the fourth inequality follows from selecting the step size such that $\alpha \leq \frac{1}{4L_i}$. Therefore, we obtain 
\begin{align*}
J^{(i)}(\widetilde{K}_0(s)) - J^{(i)}(\widetilde{K}^{(i)}_\star) &\leq \left( 1 -\frac{\alpha s \gamma_i(\widetilde{K}_0)}{4}\right)\left(J^{(i)}(\widetilde{K}_0) - J^{(i)}(\widetilde{K}^{(i)}_\star)\right) + \frac{3\alpha s}{4}b_i(\widetilde{K}_0)\\
&\leq \left(1 -\frac{\alpha s \gamma_i(\widetilde{K}_0)}{8}\right)\left(J^{(i)}(\widetilde{K}_0) - J^{(i)}(\widetilde{K}^{(i)}_\star)\right), 
\end{align*}
where the last inequality is due to the low heterogeneity regime: $b_i(\widetilde{K}_0) \leq \frac{\gamma_i(\widetilde{K}_0)\left(J^{(i)}(\widetilde{K}_0) - J^{(i)}(\widetilde{K}^{(i)}_\star)\right)}{6}$. Therefore, we have demonstrated that for a sufficiently small step size and low heterogeneity regime, there exists $\beta_1,\ldots,\beta_N \geq 1$, $\omega_1,\ldots,\omega_N >0 $, and $M_1,\ldots,M_N <\infty$, such that $\widetilde{K}_{0}(s)$ satisfies
\begin{align*}
&J^{(i)}(\widetilde{K}_{0}(s))-J^{(i)}(\widetilde{K}_\star^{(i)})\nonumber
\leq\beta_i(J^{(i)}(\widetilde{K}_0)-J^{(i)}(\widetilde{K}_\star^{(i)})),\\ 
&\leo{\rho\left(A^{(i)} - B^{(i)} \widetilde{K}_{0}(s)\leo{S^{(i)\dagger}_p}\right) < 1,}\\
&\leo{\|\widetilde{K}_{0}(s)\|_F < \infty,}\\
&\leo{\left(A^{(i)} - B^{(i)}\widetilde{K}_{0}(s) S^{(i)\dagger}_p,\Sigma^{(i)1/2}_\nu\right) \text{ is controllable}},
\end{align*}
for every $s \in [0,1]$. The proof is then completed by applying this step to every $\widetilde{K}_n$, with $n \in \mathbb{N}$. We note that the set can always be slightly enlarged so that, for a sufficiently small step size, the entire update segment remains within a region where the cost is smooth.
\end{proof}

By the above lemma, we have that every iterate $\widetilde{K}_n$ from \eqref{eq:PG_iterates_multitask} remains stabilizing for all tasks in the training set, i.e., it remains within $\Kstab$.

\vspace{0.2cm}
\noindent\textbf{Per-iterate descent on the task-specific cost.} Since each $J^{(i)}$ is $L_i$-smooth on $\Kstab$ for every task $i \in \mathcal{S}$, Lemma \ref{lem:policy_gradient_stability} implies that
$\norm{\nabla^2 J^{(i)}(\widetilde{K})} \leq L_i,$
for all lifted controllers $\widetilde{K}$ lying on the line segment between $\widetilde{K}_{n-1}$ and $\widetilde{K}_n$, and for all $n \in \setN_+$. Consequently, we obtain
\begin{align}\label{eq:thm3_proof_2}
    J^{(i)}(\widetilde{K}_n)-J^{(i)}(\widetilde{K}_{n-1}) &\leq\; \inner{\nabla J^{(i)}(\widetilde{K}_{n-1})}{\widetilde{K}_n-\widetilde{K}_{n-1}}+\frac{L_i}{2}\norm{\widetilde{K}_n-\widetilde{K}_{n-1}}_F^2 \nonumber\\
    &=\inner{\nabla J^{(i)}(\widetilde{K}_{n-1})}{-\alpha\nabla \hat{J}(\widetilde{K}_{n-1})}+\frac{L_i\alpha^2}{2}\norm{\nabla \hat{J}(\widetilde{K}_{n-1})}_F^2\nonumber\\
    &=\inner{\nabla J^{(i)}(\widetilde{K}_{n-1})}{-\alpha\nabla J^{(i)}(\widetilde{K}_{n-1})+\alpha\nabla J^{(i)}(\widetilde{K}_{n-1})-\alpha\nabla \hat{J}(\widetilde{K}_{n-1})}+\frac{L_i\alpha^2}{2}\norm{\nabla \hat{J}(\widetilde{K}_{n-1})}_F^2\nonumber\\
    &\leq -\alpha\norm{\nabla J^{(i)}(\widetilde{K}_{n-1})}_F^2+\alpha\inner{\nabla J^{(i)}(\widetilde{K}_{n-1})}{\nabla J^{(i)}(\widetilde{K}_{n-1})-\nabla \hat{J}(\widetilde{K}_{n-1})}\nonumber\\
    &+\frac{\alpha}{8}\norm{-\nabla J^{(i)}(\widetilde{K}_{n-1})+\nabla J^{(i)}(\widetilde{K}_{n-1})-\nabla \hat{J}(\widetilde{K}_{n-1})}_F^2,
\end{align}
with the last inequality due to setting the step-size as $\alpha<1/4L_i$. Moreover, by applying Young's inequality \eqref{eq:youngs} to the second term and \cite[Lemma A.1]{stamouli2024rate} to the third term on the right-hand side of \eqref{eq:thm3_proof_2}, we obtain the following expression
\begin{align}\label{eq:thm3_proof_3}
    J^{(i)}(\widetilde{K}_n)&-J^{(i)}(\widetilde{K}_{n-1})\leq  -\alpha\norm{\nabla J^{(i)}(\widetilde{K}_{n-1})}_F^2+\frac{\alpha}{2}\norm{\nabla J^{(i)}(\widetilde{K}_{n-1})}_F^2+\frac{\alpha}{2}\norm{\nabla J^{(i)}(\widetilde{K}_{n-1})-\nabla \hat{J}(\widetilde{K}_{n-1})}_F^2\nonumber\\
    &+\frac{\alpha}{4}\norm{\nabla J^{(i)}(\widetilde K_{n-1})}_F^2+\frac{\alpha}{4}\norm{\nabla J^{(i)}(\widetilde K_{n-1})-\nabla \hat{J}(\widetilde K_{n-1})}_F^2\nonumber\\
    &=-\frac{\alpha}{4}\norm{\nabla J^{(i)}(\widetilde K_{n-1})}_F^2+\frac{3\alpha}{4}\norm{\nabla J^{(i)}(\widetilde K_{n-1})-\nabla \hat{J}(\widetilde K_{n-1})}_F^2.
\end{align}

\vspace{0.2cm}
\noindent\textbf{Gradient heterogeneity via the bisimulation-based heterogeneity measure.} We now leverage Theorem \ref{theorem:Bisimulation_Functions_characterization} to write
\begin{equation*}
    \norm{\nabla J^{(i)}(\widetilde{K})-\nabla \hat{J}(\widetilde{K})}_F^2\leq b_i(\widetilde{K}),
\end{equation*}
which implies
\begin{equation}\label{eq:thm3_proof_4}
    \norm{\nabla J^{(i)}(\widetilde{K}_{n-1})-\nabla \hat{J}(\widetilde{K}_{n-1})}_F^2\leq b_i(\widetilde{K}_{n-1}).
\end{equation}

Therefore, by combining the gradient dominance property (i.e., \eqref{gradient_dominance}) with \eqref{eq:thm3_proof_3} and \eqref{eq:thm3_proof_4}, we obtain
\begin{align}\label{eq:thm3_proof_5}
    J^{(i)}(\widetilde{K}_n)-J^{(i)}(\widetilde{K}_\star^{(i)})&\leq\left(1-\frac{\leo{\alpha{\gamma_i}}(\widetilde{K}_{n-1})}{4}\right)(J^{(i)}(\widetilde{K}_{n-1})-J^{(i)}(\widetilde{K}_{\star}^{(i)}))+\frac{3\alpha}{4}b_i(\widetilde{K}_{n-1})\notag \\
    &\leq \leo{\left(1-\frac{\alpha\inf_{\widetilde{K}\in \Kstab}\gamma_i(\widetilde{K})}{4}\right)(J^{(i)}(\widetilde{K}_{n-1})-J^{(i)}(\widetilde{K}_{\star}^{(i)}))+\frac{3\alpha}{4}b_i(\widetilde{K}_{n-1})}.
\end{align}

Setting the step-size as $\alpha<4/\leo{\underline{\gamma_i}}$, with $\leo{\underline{\gamma_i}} := \inf_{\widetilde{K}\in \Kstab}\gamma_i(\widetilde{K})$, and by telescoping \eqref{eq:thm3_proof_5}, we obtain
\begin{align}\label{eq:thm3_proof_6}
    J^{(i)}(\widetilde{K}_n)-J^{(i)}(\widetilde{K}_\star^{(i)})\leq\left(1-\frac{\leo{\alpha\underline{\gamma_i}}}{4}\right)^n(J^{(i)}(\widetilde{K}_0)-J^{(i)}(\widetilde{K}_{\star}^{(i)}))+\frac{3\alpha}{4}\sum_{\ell=0}^{n-1}\left(1-\frac{\leo{\alpha\underline{\gamma_i}}}{4}\right)^{n-\ell-1}b_i(\widetilde{K}_{\ell}).
\end{align}

Let $$S_n=\sum_{\ell=0}^{n-1}\left(1-\frac{\leo{\alpha\underline{\gamma_i}}}{4}\right)^{n-\ell-1}b_i(\widetilde{K}_{\ell}).$$ Then, we can write
\begin{align*}
    S_{n+1} &= \left(1-\frac{\leo{\alpha\underline{\gamma_i}}}{4}\right)\sum_{\ell=0}^{n}\left(1-\frac{\leo{\alpha\underline{\gamma_i}}}{4}\right)^{n-\ell-1}b_i(\widetilde{K}_\ell)\\
    &= \left(1-\frac{\leo{\alpha\underline{\gamma_i}}}{4}\right)\left(S_n+\left(1-\frac{\leo{\alpha\underline{\gamma_i}}}{4}\right)^{-1}b_i(\widetilde{K}_n)\right)\\
    &=\left(1-\frac{\leo{\alpha\underline{\gamma_i}}}{4}\right) S_n+b_i(\widetilde{K}_n).
\end{align*}
By taking the limit superior on both sides of the above equality, we obtain the following:
\begin{align*}
    \limsup_{n\to\infty}S_{n+1}=\left(1-\frac{\leo{\alpha\underline{\gamma_i}}}{4}\right)\limsup_{n\to\infty}S_n+\limsup_{n\to\infty}b_i(\widetilde{K}_n)\implies \limsup_{n\to\infty}S_n=\frac{4}{\leo{\alpha\underline{\gamma_i}}}\limsup_{n\to\infty}b_i(\widetilde{K}_n).
\end{align*}

Therefore, since $1-\frac{\leo{\alpha\underline{\gamma_i}}}{4}\in(0,1)$, from \eqref{eq:thm3_proof_6} we obtain
\begin{align*}
    \limsup_{n\to\infty}(J^{(i)}(\widetilde{K}_n)-J^{(i)}(\widetilde{K}_\star^{(i)}))\leq \frac{3\alpha}{4}\limsup_{n\to\infty}S_n=\frac{3}{\leo{\underline{\gamma_i}}}\limsup_{n\to\infty}b_i(\widetilde{K}_n),
\end{align*}
which implies 
\begin{align*}
   J^{(i)}(\widetilde{K}_{\infty,\mathcal{S}})-J^{(i)}(\widetilde{K}_\star^{(i)})\leq \frac{3 \big\|\leo{S^{(i)}_p}\big\|^2\big\|\Sigma_{\widetilde{K}_\star^{(i)}}^{(i)}\big\| b_i(\widetilde{K}_{\infty,\mathcal{S}})}{4\leo{\inf_{\widetilde{K}\in\Kstab }\lambda_{\min}(\Sigma_{\widetilde{K}}^{(i)})^2} \lambda_{\min}(R^{(i)})}.
\end{align*}
when the sequence of policy gradient iterates $\{\widetilde{K}_n\}_{n}$ converges to $\widetilde{K}_{\infty,\mathcal{S}}$, which completes the proof.
\end{proof}

\section{Generalization Bound}
\label{app:generalization}

\begin{assumption}[Uniform out-of-sample cost bound]\label{assumption:uniform out=of-sample cost bound}
Conditioned on the training set $\mathcal S$, the learned controller
$\widetilde K_{\infty,\mathcal S}$ satisfies, almost surely,
$$
0\leq J^{(i)}(\widetilde K_{\infty,\mathcal S})
\leq J_{\star,\mathcal S}+\mu_{\mathcal S}b_{\mathcal S}
$$
for every $\mathcal T^{(i)}\in\operatorname{supp}(\mathcal P_{\mathcal T})$.
\end{assumption}

\subsection{Proof of Theorem \ref{theorem:generalization}}
\label{app:proof_thm_VI1}

\begin{proof}
Let $\mathcal{S}_{\text{val}}$ be the validation set with size $N_\text{val}$ and suppose that Assumption \ref{assumption:uniform out=of-sample cost bound} hold. Then for every task $\mathcal{T}^{(i)} \in \mathcal{S}_{\text{val}}$, the random variable $X^{(i)} := J^{(i)}(\Ktil_{\infty,\calS}).$ By Assumption  \ref{assumption:uniform out=of-sample cost bound}, we have
$$
  0 \leq X^{(i)} \leq J^{(i)}(\Ktil^{(i)}_\star) + \frac{3\norm{\leo{S^{(i)}_p}}^2\norm{\Sigma^{(i)}_{\Ktil^{(i)}_\star}} b_i(\Ktil_{\infty,\calS})}{4\leo{\inf_{\widetilde{K}\in\Kstab }\lambda_{\min}(\Sigma_{\widetilde{K}}^{(i)})^2}\lmin(R^{(i)})}
  \leq J_{\star,\calS} + \mu_{\calS} b_{\calS} =: B,
$$
where $J_{\star,\calS}$, $\mu_{\calS}$, $b_{\calS}$ are as defined in Section \ref{sec:generalization}. We note that, since the validation tasks $\{\mathcal{T}^{(i)}\}_{i=1}^{N_{\text{val}}}$ are drawn
i.i.d. from $\PT$, the empirical cost
$\hat{J}(\Ktil_{\infty,\calS}) = \frac{1}{N_{\text{val}}}\sum_{i=1}^{N_{\text{val}}} X^{(i)}$
is an average of bounded, i.i.d. random variables with values in $[0,B]$.

\begin{lemma}[Hoeffding's Inequality \cite{vershynin2018high}]
\label{lem:hoeffding}
Let $X^{(1)}, \dots, X^{(N)}$ be independent random variables such that, for each $i \in [N]$, $X^{(i)} \in [0, B_i]$ almost surely. Then, for any $t > 0$, we have
\begin{align}
    \mathbb{P} \left( \sum_{i=1}^N (X^{(i)} - \E[X^{(i)}]) \geq t \right)
    \leq \exp \left( - \frac{2 t^2}{\sum_{i=1}^N B^{2}_i} \right).
\end{align}
Equivalently, for the sample average $\bar{X} := \frac{1}{N}\sum_{i=1}^N X^{(i)}$,
\begin{align}
    \mathbb{P} \left( \bar{X} - \E[\bar{X}] \ge \epsilon \right)
    \leq \exp \left( - \frac{2 N^2 \epsilon^2}{\sum_{i=1}^N B^{2}_i} \right),
\end{align}
for any $\epsilon > 0$.
\end{lemma}

Leveraging the above lemma, we obtain
$$
  \mathbb{P} \left[
    \left|J(\Ktil_{\infty,\calS}) - \hat{J}(\Ktil_{\infty,\calS})\right|
    \geq \epsilon \right] \leq 2\exp \left(-\frac{2N\epsilon^2}{B^2}\right).
$$
Setting the right-hand side equal to $\delta^\prime$, solving for $\epsilon$,
and taking a union bound over the event where Assumption \ref{assump:uniform_stabilization} holds, we obtain, with probability, at least $1-\delta^\prime-\delta$,
$$
  \left|J(\Ktil_{\infty,\calS}) - \hat{J}(\Ktil_{\infty,\calS})\right|
  \leq B\sqrt{\frac{\log(4/(\delta^\prime+\delta))}{2N_{\text{val}}}} = (J_{\star,\calS}+\mu_{\calS}b_{\calS})\sqrt{\frac{\log(4/(\delta^\prime+\delta))}{2N_{\text{val}}}},
$$
which completes the proof.
\end{proof}

\section{Model-Free Analysis}
\label{app:model_free}

\subsection{Proof of Lemma \ref{lemma:variance_red}}
\label{app:proof_lemma_VII1}

We now analyze the one-point zeroth-order gradient estimator
$\hat{\nabla}\hat{J}(\Ktil)$ defined in \eqref{gradient_estimation} and prove the claimed bound on estimation error. For the model-free analysis, we additionally assume that the process and measurement noises are almost surely bounded.

\begin{proof}
Recall the gradient estimator for each task $\mathcal{T}^{(i)} \in \mathcal{S}$
$$
\widehat{\nabla}J^{(i)}(\Ktil) := \frac{d}{n_s}\sum_{m=1}^{n_s}\frac{\hat{J}^{(i)}_\tau(\Ktil+U_{i,m})U_{i,m}}{r^2},
$$
with $d:=n_u p(n_u+n_y)$. Here, we recall that $n_s$ is the number
of samples, $r$ is the smoothing radius, and each $U_{i,m}$ is drawn uniformly from the distribution of matrices with $\normF{U_{i,m}}=r$.

\vspace{0.2cm}

We first note that conditioned on the noise sequence $\xi^l$, we have that 
$$
  \E\left[\widehat{\nabla}J^{(i)}(\Ktil)\right]
  = \nabla J^{(i)}_r(\Ktil),
$$
where $J^{(i)}_r(\Ktil):=\E_{U}[J^{(i)}(\Ktil+U)]$ is the
smoothed cost.  Moreover, we can also write
\begin{align*}
      \normF{\nabla J^{(i)}_r(\Ktil)-\nabla J^{(i)}(\Ktil)} &= \normF{\nabla \E_{U}[J^{(i)}(\Ktil+U)]-\nabla J^{(i)}(\Ktil)}\\
  &= \normF{ \E_{U} \nabla J^{(i)}(\Ktil+U)-\nabla J^{(i)}(\Ktil)}\\
  &\leq  \E_{U}\normF{ \nabla J^{(i)}(\Ktil+U)-\nabla J^{(i)}(\Ktil)}\\ &\leq  L_i \|\Ktil+U - \Ktil\|_F = L_i \|U\|_F = L_i r.
\end{align*}
where we use the definition of $J^{(i)}_r(\Ktil)$ and the smoothness of the cost gradient.

In addition, since $J^{(i)}(\Ktil)\leq\bar{J}_i$, for any stabilizing $\widetilde{K}$, we can write task-specific estimation variance as follows:
$$
  \E\left[\normF{\widehat{\nabla}J^{(i)}(\Ktil)-\nabla J^{(i)}_r(\Ktil)}^2\right]
  \leq \frac{4\bar{J}^2_i d^2  r^2}{n_s r^4}
  = \frac{4\bar{J}^2_i d^2}{n_s r^2}.
$$

We then recall that the multitask estimator averages over $N$ tasks in the training set, i.e.,  $\hat{\nabla}\hat{J}(\Ktil)=\frac{1}{N}\sum_{i=1}^N\hat{\nabla}J^{(i)}(\Ktil)$. Since the tasks are independent, we can write 
$$
\E \left[\normF{\widehat{\nabla}\hat{J}(\Ktil)
            -\nabla\hat{J}(\Ktil)}^2\right] \leq \frac{1}{N^2}\sum_{i=1}^N\frac{4\bar{J}^2_i d^2}{n_s r^2} \leq \frac{4\bar{J}^2 d^2}{n_s N r^2},
$$
with $\bar{J}:=\max_{\mathcal{T}^{(i)} \in \mathcal{S}}\bar{J}_i$.  The variance therefore scales as $1/(n_s N)$, confirming the $1/N$ variance reduction in the multitask setting.

\begin{lemma}[Matrix Hoeffding inequality \cite{tropp2012user}]
\label{lem:hoeffding_frobenius_matrix_average}
Let $X^{(1)},\dots,X^{(N)} \in \mathbb{R}^{d_1 \times d_2}$ be independent random matrices such that
\begin{align*}
    \E[X^{(i)}] = 0,
\end{align*}
for all $i \in [N]$. Assume that there exist deterministic constants $R_1,\dots,R_N > 0$ such that
\begin{align*}
    \norm{X^{(i)}}_F \leq R_i,
\end{align*}
almost surely for all $i \in [N]$. Then, for every $t > 0$,
\begin{align*}
    \mathbb{P}\left(
        \left\| \frac{1}{N}\sum_{i=1}^N X^{(i)} \right\|_F \geq t
    \right)\leq 2 \exp \left(
        - \frac{N^2 t^2}{2 \sum_{i=1}^N R^{2}_i}
    \right).
\end{align*}
Equivalently, for any $\widetilde \delta \in (0,1)$, with probability at least $1-\widetilde \delta$,
\begin{align*}
\left\| \frac{1}{N}\sum_{i=1}^N X^{(i)} \right\|_F \leq
\frac{1}{N}\sqrt{2 \log \left(\frac{2}{\widetilde \delta}\right)\sum_{i=1}^N R_i^2 }.
\end{align*}
In particular, if $\norm{X^{(i)}}_F \leq R$ almost surely for all $i \in [N]$, then for every $t>0$,
\begin{align*}
    \mathbb{P} \left(
\left\| \frac{1}{N}\sum_{i=1}^N X^{(i)} \right\|_F \geq t \right)
\leq 2 \exp \left(- \frac{N t^2}{2 R^2}\right),
\end{align*}
and, with probability at least $1-\widetilde \delta$,
\begin{align*}
\left\| \frac{1}{N}\sum_{i=1}^N X^{(i)} \right\|_F \leq R \sqrt{\frac{2 \log \left(\frac{2}{\widetilde \delta}\right)}{N}}.
\end{align*}
\end{lemma}

Note that for each task $\mathcal{T}^{(i)} \in \mathcal{S}$ and sample $m \in [n_s]$, we can define
\begin{align*}
Z_{i,m} := \frac{d}{r^2} J_{\tau}^{(i)}(\widetilde K + U_{i,m}) U_{i,m}, \text{ with } X_{i,m} := Z_{i,m} - \E[Z_{i,m}].
\end{align*}

Then $\mathbb{E}[X_{i,m}] = 0$, the matrices $\{X_{i,m}\}_{i,m}$ are independent, and
\begin{align*}
\widehat{\nabla} J(\widetilde K) - \nabla  J_r(\widetilde K) = \frac{1}{Nn_s}\sum_{i=1}^N \sum_{m=1}^{n_s} X_{i,m},
\end{align*}
with $J_r(\widetilde K) = \frac{1}{N}\sum_{i=1}^N J^{(i)}_r(\widetilde{K})$. Moreover, using $\|U_{i,m}\|_F = r$ and $J_{\tau}^{(i)}(\cdot)\leq \bar J$, we obtain
\begin{align*}
    \|X_{i,m}\|_F \leq \|Z_{i,m}\|_F + \|\E Z_{i,m}\|_F \leq \|Z_{i,m}\|_F +\E \| Z_{i,m}\|_F \leq \frac{2d\bar J}{r}.
\end{align*}

By applying Lemma \ref{lem:hoeffding_frobenius_matrix_average} to the average of the $Nn_s$ independent centered matrices yields that, with probability at least $1-\widetilde\delta$,
\begin{align*}
\left\| \widehat{\nabla} J(\widetilde K) - \nabla  J_r(\widetilde K)
\right\|_F^2 \leq  \frac{8 \bar J^2 d^2 \log(2/\widetilde\delta)}{Nn_s r^2},
\end{align*}
Therefore, since
\begin{align*}
    \left\| \nabla \bar J_r(\widetilde K) - \nabla \bar J(\widetilde K)
    \right\|_F \leq Lr,
\end{align*}
we further obtain
\begin{align*}
\left\| \widehat{\nabla}\bar J(\widetilde K) - \nabla \bar J(\widetilde K)\right\|_F^2\leq\frac{16\bar J^2 d^2 \log(2/\widetilde\delta)}{Nn_s r^2}+2L^2 r^2.
\end{align*}
Choosing the smoothing radius as follows:
\begin{align*}
r=\sqrt{\frac{d\bar J}{L}}\left(\frac{Nn_s}{\log(2/\widetilde\delta)}\right)^{-1/4}
\end{align*}
balances the two terms and leads to the following expression for the estimation error:
\begin{align*}
\left\|\widehat{\nabla}\bar J(\widetilde K) - \nabla \bar J(\widetilde K)\right\|_F^2 \leq 18Ld\bar J\sqrt{\frac{\log(2/\widetilde\delta)}{Nn_s}}.
\end{align*}
\end{proof}

\subsection{Proof of Corollary \ref{cor:model-free}}
\label{app:proof_cor1}

\begin{proof}
We begin by writing the per-iterate descent on the task-specific cost under the gradient estimation. That is,

\begin{align}
    J^{(i)}(\widetilde{K}_n)-J^{(i)}(\widetilde{K}_{n-1}) &\leq\; \inner{\nabla J^{(i)}(\widetilde{K}_{n-1})}{\widetilde{K}_n-\widetilde{K}_{n-1}}+\frac{L_i}{2}\norm{\widetilde{K}_n-\widetilde{K}_{n-1}}_F^2 \nonumber\\
    &=\inner{\nabla J^{(i)}(\widetilde{K}_{n-1})}{-\alpha\widehat \nabla \hat{J}(\widetilde{K}_{n-1})}+\frac{L_i\alpha^2}{2}\norm{\widehat \nabla \hat{J}(\widetilde{K}_{n-1})}_F^2\nonumber\\
    &=\inner{\nabla J^{(i)}(\widetilde{K}_{n-1})}{-\alpha\nabla J^{(i)}(\widetilde{K}_{n-1})+\alpha\nabla J^{(i)}(\widetilde{K}_{n-1})-\alpha \widehat \nabla \hat{J}(\widetilde{K}_{n-1})}+\frac{L_i\alpha^2}{2}\norm{\widehat \nabla \hat{J}(\widetilde{K}_{n-1})}_F^2\nonumber\\
    &\leq -\alpha\norm{\nabla J^{(i)}(\widetilde{K}_{n-1})}_F^2+\alpha\inner{\nabla J^{(i)}(\widetilde{K}_{n-1})}{\nabla J^{(i)}(\widetilde{K}_{n-1})-\widehat \nabla \hat{J}(\widetilde{K}_{n-1})}\nonumber\\
    &+\frac{\alpha}{8}\norm{-\nabla J^{(i)}(\widetilde{K}_{n-1})+\nabla J^{(i)}(\widetilde{K}_{n-1})-\widehat \nabla \hat{J}(\widetilde{K}_{n-1})}_F^2,
\end{align}
where the last inequality follows from setting the step-size according to $\alpha<1/4L_i$. Moreover, following the proof of Theorem \ref{theorem:algorithm_dependent}, we also apply Young's inequality to the second term and \cite[Lemma A.1]{stamouli2024rate} to the third term on the right-hand side of the above expression to obtain
\begin{align}\label{eq:thm3_proof_3_modelfree}
    J^{(i)}(\widetilde{K}_n)-J^{(i)}(\widetilde{K}_{n-1}) &\leq  -\alpha\norm{\nabla J^{(i)}(\widetilde{K}_{n-1})}_F^2+\frac{\alpha}{2}\norm{\nabla J^{(i)}(\widetilde{K}_{n-1})}_F^2+\frac{\alpha}{2}\norm{\nabla J^{(i)}(\widetilde{K}_{n-1})-\widehat \nabla \hat{J}(\widetilde{K}_{n-1})}_F^2\nonumber\\
    &+\frac{\alpha}{4}\norm{\nabla J^{(i)}(\widetilde K_{n-1})}_F^2+\frac{\alpha}{4}\norm{\nabla J^{(i)}(\widetilde K_{n-1})-\widehat \nabla \hat{J}(\widetilde K_{n-1})}_F^2\nonumber\\
    &=-\frac{\alpha}{4}\norm{\nabla J^{(i)}(\widetilde K_{n-1})}_F^2+\frac{3\alpha}{4}\norm{\nabla J^{(i)}(\widetilde K_{n-1})-\widehat \nabla \hat{J}(\widetilde K_{n-1})}_F^2\notag\\
    &\leq -\frac{\alpha}{4}\norm{\nabla J^{(i)}(\widetilde K_{n-1})}_F^2+\underbrace{\frac{3\alpha}{2}\norm{\nabla \hat{J}(\widetilde K_{n-1}) -\widehat \nabla \hat{J}(\widetilde K_{n-1})}_F^2}_{\text{gradient estimation error}} + \frac{3\alpha}{2}\norm{\nabla J^{(i)}(\widetilde K_{n-1})-\nabla \hat{J}(\widetilde K_{n-1})}_F^2\notag\\
    &\leq -\frac{\alpha}{4}\norm{\nabla J^{(i)}(\widetilde K_{n-1})}_F^2+27\alpha Ld\bar{J} \sqrt{\frac{\log(2/\widetilde\delta)}{Nn_s}} + \frac{3\alpha}{2}b_i(\widetilde K_{n-1})
\end{align}

Therefore, by invoking the gradient dominance property \eqref{gradient_dominance}, we can write 

\begin{align*}
    J^{(i)}(\widetilde{K}_n)-J^{(i)}(\widetilde{K}^{(i)}_\star) \leq \left( 1 -\frac{\alpha \leo{\underline{\gamma_i}}}{4}\right)\left(J^{(i)}(\widetilde{K}_{n-1})-J^{(i)}(\widetilde{K}^{(i)}_\star)\right) +27\alpha Ld\bar{J} \sqrt{\frac{\log(2/\widetilde\delta)}{Nn_s}} + \frac{3\alpha}{2}b_i(\widetilde K_{n-1}),
\end{align*}
and by unrolling the above expression over iterations, we obtain
\begin{align*}
J^{(i)}(\widetilde{K}_n)-J^{(i)}(\widetilde{K}^{(i)}_\star) \leq \left( 1 -\frac{\alpha \leo{\underline{\gamma_i}}}{4}\right)^n\left(J^{(i)}(\widetilde{K}_{0})-J^{(i)}(\widetilde{K}^{(i)}_\star)\right) +\frac{108 Ld\bar{J}}{\leo{\underline{\gamma_i}}} \sqrt{\frac{\log(2/\widetilde\delta)}{Nn_s}} + \frac{3\alpha}{2} \underbrace{\sum_{\ell=0}^{n-1}\left(1-\frac{\leo{\alpha\underline{\gamma_i}}}{4}\right)^{n-\ell-1}b_i(\widetilde{K}_{\ell})}_{S_n},
\end{align*}
and by using the same analysis for $S_n$ as in the proof of Theorem \ref{theorem:algorithm_dependent}, the task-specific cost of the asymptotic controller obtained from the model-free policy gradient updates satisfies
$$
J^{(i)}(\Ktil_{\infty,\calS}) - J^{(i)}(\Ktil^{(i)}_\star) \leq 
  \underbrace{ \frac{3\norm{\leo{S^{(i)}_p}}^2\norm{\Sigma^{(i)}_{\Ktil^{(i)}_\star}}
        \,b_i(\Ktil_{\infty,\calS})}       {2\leo{\inf_{\widetilde{K}\in\Kstab }\lambda_{\min}(\Sigma_{\widetilde{K}}^{(i)})^2}\lmin(R^{(i)})}}_{\text{heterogeneity bias}}
+\underbrace{\frac{108\,Ld\bar{J}}{\leo{\underline{\gamma_i}}}
      \sqrt{\frac{\log(2/\widetilde\delta)}{n_s N}}}_{\text{estimation error}}.
$$
The first term is the irreducible bias from task heterogeneity
(present even with infinitely many samples). The second term comes from the variance of the zeroth-order gradient estimator, which decays as
$(n_s N)^{-1/2}$, making explicit the benefit of multitask learning, namely, for fixed number of trajectories $n_s$, every additional training task reduces the estimation error by a factor of $1/\sqrt{N}$.
\end{proof}

\end{document}